\documentclass[12pt]{amsart}

\usepackage[letterpaper,margin=1.1in]{geometry}
\usepackage{amsmath,amsthm,amssymb,amsfonts,graphicx,mathrsfs,bm,amscd,latexsym,color,enumerate} 

\newtheorem{maintheorem}{Theorem}
\newtheorem{maincor}{Corollary}

\newtheorem{theorem}{Theorem}[section]
\newtheorem{proposition}[theorem]{Proposition}
\newtheorem{lemma}[theorem]{Lemma}
\newtheorem{corollary}[theorem]{Corollary}

\theoremstyle{definition}
\newtheorem{definition}[theorem]{Definition}

\theoremstyle{remark}
\newtheorem{remark}[theorem]{Remark}
\newtheorem{problem}[theorem]{Problem}
\newtheorem{example}[theorem]{Example}

\numberwithin{figure}{section}

\def\N{\mathbb{N}}
\def\e{\varepsilon}
\def\g{\gamma}
\def\a{\alpha}
\def\d{\delta}
\def\D{\Delta}

\def\G{\Gamma}

\def\C{\mathcal{C}}
\def\R{\mathbb{R}}
\def\Z{\mathbb{Z}}

\def\W{\mathcal{W}}

\newcommand{\XX}{\mathbb{X}}

\newcommand{\RR}{\mathbb{R}}
\newcommand{\Haus}{\mathcal{H}}
\newcommand{\Pack}{\mathcal{P}}
\newcommand{\dist}{\mathop\mathrm{dist}\nolimits}
\newcommand{\diam}{\mathop\mathrm{diam}\nolimits}

\newcommand{\res}{\hbox{ {\vrule height .22cm}{\leaders\hrule\hskip.2cm} }} 

\newcommand{\leaves}{\mathsf{Leaves}}
\newcommand{\tube}{\mathsf{Tube}}
\newcommand{\entrance}{\mathsf{Entrance}}
\newcommand{\exit}{\mathsf{Exit}}
\newcommand{\side}{\mathsf{Side}}
\newcommand{\Top}{\mathop\mathsf{Top}}

\newcommand{\Hold}{\mathop\mathrm{H\ddot{o}ld}\nolimits}
\newcommand{\gap}{\mathop\mathrm{gap}\nolimits}

\newcommand{\card}{\mathop\mathrm{card}\nolimits}

\numberwithin{equation}{section}

\begin{document}

\title[Geometry of measures in real dimensions]{Geometry of measures in real dimensions\\ via H\"older parameterizations}
\thanks{Badger was partially supported by NSF grants 1500382 and 1650546. Part of this work was carried out while Badger attended the long program on \emph{Harmonic Analysis} at MSRI in Spring 2017.}
\date{March 26, 2018}
\author{Matthew Badger \and Vyron Vellis}
\address{Department of Mathematics\\ University of Connecticut\\ Storrs, CT 06269-1009}
\email{matthew.badger@uconn.edu}
\address{Department of Mathematics\\ University of Connecticut\\ Storrs, CT 06269-1009}
\email{vyron.vellis@uconn.edu}

\begin{abstract}We investigate the influence that $s$-dimensional lower and upper Hausdorff densities have on the geometry of a Radon measure in $\RR^n$ when $s$ is a real number between $0$ and $n$. This topic in geometric measure theory has been extensively studied when $s$ is an integer. In this paper, we focus on the non-integer case, building upon a series of papers on $s$-sets by Mart\'in and Mattila from 1988 to 2000. When $0<s<1$, we prove that measures with almost everywhere positive lower density and finite upper density are carried by countably many \emph{bi-Lipschitz curves}. When $1\leq s<n$, we identify conditions on the lower density that ensure the measure is either carried by or singular to \emph{$(1/s)$-H\"older curves}. The latter results extend part of the recent work of Badger and Schul, which examined the case $s=1$ (Lipschitz curves) in depth. Of further interest, we introduce H\"older and bi-Lipschitz parameterization theorems for Euclidean sets with ``small'' Assouad dimension.
\end{abstract}

\maketitle

\vspace{-.2cm}

\setcounter{tocdepth}{1}
\tableofcontents

\renewcommand{\thepart}{\Roman{part}}

\section{Introduction}

The study of the measure-theoretic geometry of Euclidean sets of integral dimension was initiated by Besicovitch \cite{Bes28,Bes38} in the 1920s and 1930s. Among many  original results, Besicovitch proved that any $1$-set $E\subseteq\RR^2$ (that is, an $\Haus^1$-measurable set with $0<\Haus^1(E)<\infty$) decomposes into a regular set, $E_r$, and an irregular set, $E_{pu}$, where $$\lim_{r\downarrow 0} \frac{\Haus^1(E_r\cap B(x,r))}{2r}=1\quad\text{for $\Haus^1$-a.e. $x\in E_r$}$$ and $$\liminf_{r\downarrow 0} \frac{\Haus^1(E_{pu}\cap B(x,r))}{2r} \leq 3/4\quad\text{for $\Haus^1$-a.e. $x\in E_{pu}$}.$$ Throughout this paper, $\mathcal{H}^s$ denotes $s$-dimensional Hausdorff measure (see \S\ref{s:proof-A} below). When $s=1$, $\Haus^1$ extends the Lebesgue measure of subsets of the line to a measure of ``length" on arbitrary subsets of $\RR^n$. The result quoted above says that there is a strict gap between \emph{measure-theoretic densities} on regular and irregular sets, where the Lebesgue density theorem holds precisely for regular sets. Besicovitch established several striking characterizations of regular and irregular sets in terms of \emph{global geometry} (intersection with curves, projections onto lines) and \emph{asymptotic geometry} (existence of tangent lines) of sets. In particular, in the first category of results, Besicovitch proved that regular sets are subsets of countable unions of rectifiable curves plus a set of $\Haus^1$ measure zero, whereas irregular sets are sets which intersect any rectifiable curve in a set of $\Haus^1$ measure zero. For a modern presentation of Besicovitch's theory of 1-sets, see \cite[Chapter 3]{Falconer}.

Extensions of Besicovitch's program have now been made in three different directions: expanding the range of dimensions, broadening the class of measures, and developing quantitative analogues of the qualitative theory. The decomposition of $m$-sets in $\RR^n$ for arbitrary pairs of integers $1\leq m\leq n-1$ and the characterizations of regular and irregular sets in terms of measure-theoretic densities, global geometry, and asymptotic geometry were established over a number of years, with principal contributions by Federer \cite{Fed47}, Marstrand \cite{Marstrand61}, Mattila \cite{Mattila75}, and Preiss \cite{Preiss}. Analogues of these results for a larger class of ``absolutely continuous" measures satisfying $\mu\ll\Haus^m$ were developed by Morse and Randolph \cite{MR44}, when $m=1$, and by Preiss \cite{Preiss}, when $2\leq m\leq n-1$. Very recently, an extension of Morse and Randolph's results to arbitrary Radon measures in $\RR^n$ was completed by Badger and Schul \cite{BS3}. For related recent developments in this direction, see \cite{AT15}, \cite{ENV}, and \cite{ghinassi}. A parallel quantitative theory of Ahlfors regular $m$-sets (and further results) was developed by Jones \cite{Jones-TST} and Okikiolu \cite{OK-TST}, when $m=1$, and extensively by David and Semmes \cite{DS91,DS93}, when $1\leq m\leq n-1$. A generalization of Jones' and Okikiolu's traveling salesman theorems, which identify subsets of rectifiable curves in $\RR^n$, to a theorem identifying subsets of certain higher dimensional surfaces in $\RR^n$ has recently been furnished by Azzam and Schul \cite{AS-TST}.

In \cite{MM1988,MM1993,MM2000}, Mart\'in and Mattila initiated a Besicovitch-style study of measure-theoretic densities, global geometry, and asymptotic geometry of $s$-sets $E\subseteq \RR^n$, with $0<s<n$ not necessarily an integer.\footnote{A related investigation on the Hausdorff dimension of projections of $s$-sets onto lower dimensional subspaces was carried out earlier by Marstrand \cite{Marstrand54}.}  One important finding in \cite{MM1988} is that several geometric properties, which each characterize regular $s$-sets when $s$ is an integer, no longer describe the same classes of $s$-sets for fractional $s$.

\begin{definition}[rectifiability of $s$-sets] \label{def:s-rect} Let $1\leq m\leq n-1$ be integers and $0\leq s\leq m$. Let $E\subseteq\RR^n$ be an $s$-set (i.e.~an $\Haus^s$ measurable set with $0<\Haus^s(E)<\infty$). We say that \begin{enumerate}
\item $E$ is \emph{countably $(\Haus^s,m)$ rectifiable} if there exist countably many Lipschitz maps $f_i:\RR^m\rightarrow\RR^n$ such that f $\Haus^s(E\setminus\bigcup_i f_i(\RR^m))=0$;
\item $E$ is \emph{countably $(\Haus^s,m)$ graph rectifiable} if there are countably many $m$-dimensional Lipschitz graphs $\Gamma_i\subseteq\RR^n$ (that is, isometric copies of graphs of Lipschitz functions $g:\RR^m\rightarrow\RR^{n-m}$) such that $\Haus^s(E\setminus\bigcup_i \Gamma_i)=0$; and,
\item $E$ is \emph{countably $(\Haus^s,m)$ $C^1$ rectifiable} if there exist countably many $m$-dimensional embedded $C^1$ submanifolds $M_i\subseteq\RR^n$ such that $\Haus^s(E\setminus \bigcup_i M_i)=0$.
\end{enumerate} \end{definition}

From the definition it is immediate that $(3)\Rightarrow (2)\Rightarrow (1)$ for every $s$-set. When $s=m$, the three variations of rectifiability in Definition 1.1 are in fact equivalent and furthermore hold if and only if $E$ is regular in the sense that $$\lim_{r\downarrow 0} \frac{\Haus^m(E\cap B(x,r))}{\omega_mr^m}=1\quad\text{at $\Haus^m$-a.e.~$x\in E$},$$ where $\omega_m$ is the $m$-dimensional Hausdorff measure on the unit ball in $\RR^m$ (see \cite{Mattila}). However, Mart\'in and Mattila (see \cite[\S\S5.3 and 5.4]{MM1988}) constructed compact $s$-sets in the plane that show $(1)\not\Rightarrow(2)\not\Rightarrow (3)$ when $0<s<1$.

Another principal result from \cite{MM1988} is that $s$-sets with positive lower density are always countably $(\Haus^s,m)$ rectifiable when $s<m$. This is in stark contrast with the situation when $s\geq m$; see the discussion between Theorems \ref{t:mm-den} and \ref{t:mm-pu} below.

\begin{theorem}[{\cite[Theorem 4.1(1)]{MM1988}}]\label{t:mm88} Let $1\leq m\leq n-1$ be integers and let $s<m$. If $E\subseteq\RR^n$ is an $s$-set and $$\liminf_{r\downarrow 0} \frac{\Haus^s(E\cap B(x,r))}{r^s}>0\quad\text{at $\Haus^s$-a.e.~$x\in E$},$$ then $E$ is countably $(\Haus^s,m)$ rectifiable.\end{theorem}

In \cite{MM1993,MM2000}, Mart\'in and Mattila explore a notion of fractional rectifiability based on images of H\"older continuous maps. Recall that a map $f:A\rightarrow \RR^n$ defined over $A\subseteq\RR^m$ is \emph{$(1/\gamma)$-H\"older} for some $1\leq \gamma<\infty$ if $$\Hold_{1/\gamma} f:= \sup_{\stackrel{x,y\in A}{x\neq y}}\frac{|f(x)-f(y)|}{|x-y|^{1/\gamma}}<\infty.$$ Every $(1/\gamma)$-H\"older map defined over $A\subseteq\RR^m$ admits an extension to a $(1/\gamma)$-H\"older map defined over $\RR^m$, and trivially, restrictions of $(1/\gamma)$-H\"older maps are $(1/\gamma)$-H\"older. It is also well known that a $(1/\gamma)$-H\"older map does not increase Hausdorff dimension of a set by more than a factor of $\gamma$, and moreover, $$\Haus^{\gamma t}(f(B))\lesssim \Haus^t(B)\quad\text{for all $t\geq 0$ and $B\subseteq A$,}$$ where the implicit constant depends only on $(\Hold_{1/\gamma} f)^t$ and the normalization used in the definition of the Hausdorff measures $\Haus^t$ and $\Haus^{\gamma t}$. In particular, for any integers $1\leq m\leq n-1$ and $s\in[m,n]$, images of $(m/s)$-H\"older maps $f:[0,1]^m\rightarrow\RR^n$---which to shorten terminology,  we call \emph{$(m/s)$-H\"older $m$-cubes}---are connected, compact sets with finite $\Haus^s$ measure. Where Besicovitch used rectifiable curves (images of Lipschitz maps $f:[0,1]\rightarrow\RR^n$) as a basis for studying the structure of 1-sets, Mart\'in and Mattila use $(m/s)$-H\"older $m$-cubes as a basis to examine the structure of $s$-sets when $s\in[m,n]$.

From general considerations (see Appendix \ref{a:A}), it follows that for every $s$-set $E\subseteq\RR^n$, with $s\in[m,n]$, we can write $E$ as $$E=E_{m\rightarrow s}\cup E_{m\rightarrow s}^\perp, \quad \text{with } \Haus^s(E_{m\rightarrow s}\cap E_{m\rightarrow s}^\perp)=0,$$ where \begin{itemize}
\item $E_{m\rightarrow s}$ is \emph{countably $(\Haus^s,m\rightarrow s)$ rectifiable} in the sense that $E_{m\rightarrow s}$ is covered up to a set of $\Haus^s$ measure zero by countably many $(m/s)$-H\"older $m$-cubes, and
\item $E_{m\rightarrow s}^\perp$ is \emph{purely $(\Haus^s,m\rightarrow s)$ unrectifiable} in the sense that $E_{m\rightarrow s}^\perp$ intersects any $(m/s)$-H\"older $m$-cube in a set of $\Haus^s$ measure zero.
\end{itemize} The decomposition of $E$ into its countably $(\Haus^s,m\rightarrow s)$ rectifiable and purely $(\Haus^s,m\rightarrow s)$ unrectifiable parts is unique up to redefinition of the parts on sets of $\Haus^s$ measure zero. Note that, when $s=m$, the countably $(\Haus^m,m\rightarrow m)$ rectifiable $m$-sets are precisely the countably $(\Haus^m,m)$ rectifiable $m$-sets, which are by now well understood (e.g.~see \cite{Mattila}). It is an open problem to characterize countably $(\Haus^s,m\rightarrow s)$ rectifiable $s$-sets in terms of projections, measure-theoretic densities, and/or asymptotic geometry when $s>m$.

\begin{figure}\begin{center}\includegraphics[width=.58\textwidth]{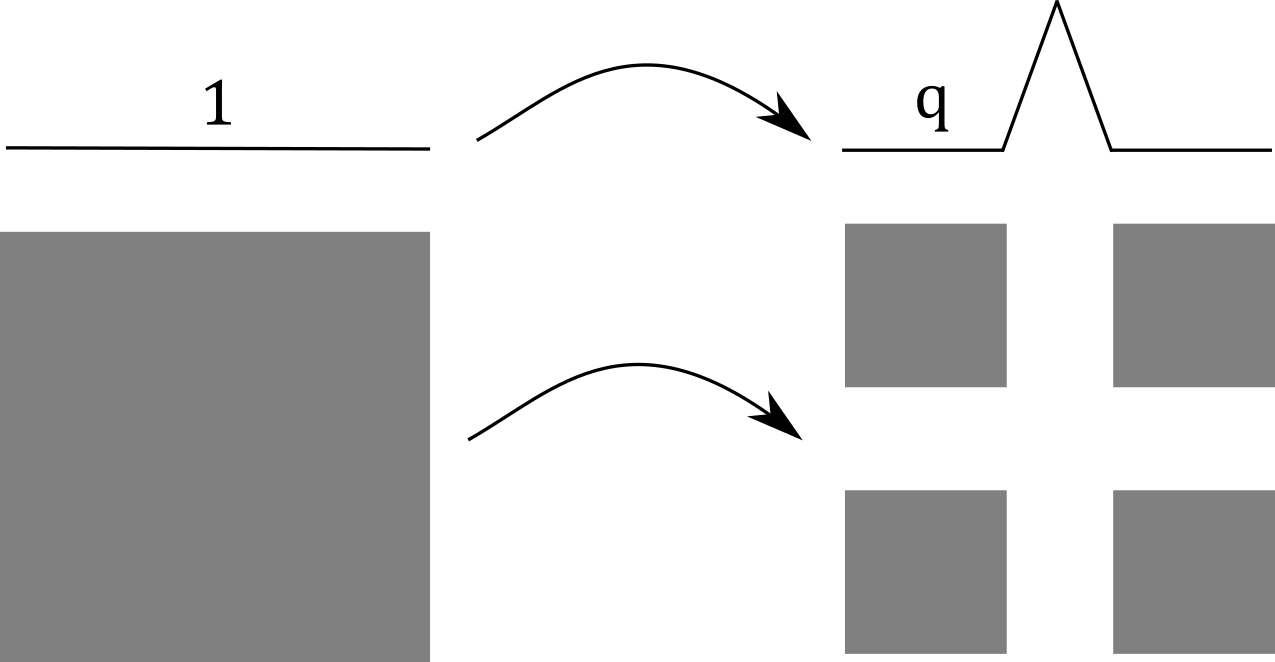}\end{center}
\caption{Generators for a snowflake curve and  four-corner Cantor set of Hausdorff dimension $s$, where $q=4^{-1/s}$.}\label{fig:snow}
\end{figure}

\begin{example}[snowflake curves] Let $\Gamma_s\subseteq\RR^2$ be a self-similar snowflake curve of Hausdorff dimension $1<s<2$ (see Figure \ref{fig:snow}). Then there exists a $(1/s)$-H\"older homeomorphism $[0,1]\rightarrow\Gamma_s$ (e.g. see \cite[Chapter VII, \S2]{ss-reals}). In particular, $\Gamma_s$ is $(\Haus^s,1\rightarrow s)$ rectifiable. \end{example}

\begin{example}[space-filling curves] The existence of H\"older space-filling (Peano) curves is well known. For example, see \cite[Chapter VII, \S3]{ss-reals} for a friendly exposition of the construction of a $(1/2)$-H\"older surjection $[0,1]\twoheadrightarrow[0,1]^2$. More generally, there exist $(j/k)$-H\"older surjections $[0,1]^j\twoheadrightarrow[0,1]^k$ for any pair of integer dimensions $1\leq j\leq k$; this follows by taking scaling limits of Stong's $(j/k)$-H\"older bijections $\Z^j\rightarrow\Z^k$ from \cite{stong} (for an outline of the argument, see \cite[\S9.1]{semmes-buffalo}). Thus, in the language above, a $k$-dimensional cube $[0,1]^k\times\{0\}^{n-k}$ is $(\Haus^k,j\rightarrow k)$ rectifiable for all $1\leq j\leq k\leq n$. Precomposing Lipschitz maps $g:[0,1]^k\rightarrow \RR^n$ with the space-filling map $[0,1]^j\twoheadrightarrow[0,1]^k$, one sees that every countably $(\Haus^k,k)$ rectifiable set $E\subseteq\RR^n$ is countably $(\Haus^k,j\rightarrow k)$ rectifiable whenever $1\leq j\leq k$.
\end{example}

The following theorem from \cite{MM1993} provides a necessary condition on the lower density for an $s$-set to be countably $(\Haus^s,m\rightarrow s)$ rectifiable.

\begin{theorem}[{\cite[Theorem 3.2]{MM1993}}] \label{t:mm-den} Let $1\leq m\leq n-1$ be integers and let $s\in[m,n]$, let $A\subseteq\RR^m$ be $\Haus^m$ measurable, and let $f:A\rightarrow\RR^n$ be $(m/s)$-H\"older. If $E\subseteq f(A)$ is $\Haus^s$ measurable, then $$\liminf_{r\downarrow 0}\frac{\Haus^s(E\cap B(x,r))}{r^s}>0\quad\text{for $\Haus^s$-a.e. $x\in E$}.$$\end{theorem}

When $s=m$, it is well known that not all $m$-sets with positive lower $m$-density are countably $(\Haus^m,m)$ rectifiable. For example, in \cite[\S5.4]{hutchinson}, Hutchinson proved that $m$-dimensional self-similar sets $K$ with disjoint parts have positive lower density\footnote{In fact, $K$ is Ahlfors regular in the sense that $\Haus^m(K\cap B(x,r))\sim r^m$ when $x\in K$ and $0<r\leq \diam K$, because $K$ is self-similar.}, but intersect images of Lipschitz maps $f:\RR^m\rightarrow\RR^n$ in sets of $\Haus^m$ measure zero. In \cite{MM1993},
Mart\'in and Mattila confirm that this behavior persists for self-similar $s$-sets when $s>m$. Also see \cite{MM2000}, where Theorem \ref{t:mm-pu} is extended to more general sets, including cylinders $K\times \RR^k$, where $K$ is a self-similar Cantor set.

\begin{theorem}[{\cite[Corollary 3.5]{MM1993}}] \label{t:mm-pu} Let $1\leq m\leq n-1$ be integers and let $s\in[m,n]$. Let $K$ be a compact self-similar subset of $\RR^n$, $K=\bigcup_{i=1}^N S_i(K)$, such that the different parts $S_i(K)$ are disjoint. If $A\subseteq\RR^m$ is $\Haus^m$ measurable and $f:A\rightarrow\RR^n$ is $(m/s)$-H\"older, then $$\Haus^s(K\cap f(A))=0.$$\end{theorem}

In this paper, we redevelop Mart\'in and Mattila's framework in the general setting of Radon measures in $\RR^n$. In particular, we investigate the connection between lower and upper Hausdorff densities of a measure and its interaction with $(m/s)$-H\"older $m$-cubes. The main results and methods will be described momentarily. Before continuing, we first record an extension of Theorem \ref{t:mm88}, which follows from Theorems B and C.

\begin{theorem}\label{t:HBC} Let $1\leq m\leq n-1$ be integers, let $s\in[m,n]$, and let $t\in[0,s)$. If $E\subseteq\RR^n$ is a $t$-set and $$\liminf_{r\downarrow 0} \frac{\Haus^t(E\cap B(x,r))}{r^t}>0\quad\text{at $\Haus^t$-a.e.~$x\in E$},$$ then $E$ is countably $(\Haus^t,m\rightarrow s)$ rectifiable, i.e.~ there exist countably many $(m/s)$-H\"older $m$-cubes $\Gamma_i$ such that $\Haus^t(E\setminus \bigcup_i \Gamma_i)=0$.

Moreover, if $t\in[0,1)$, then $E$ is countably $(\Haus^t,1)$ bi-Lipschitz rectifiable, i.e.~ there exist countably many bi-Lipschitz embeddings $f_i:[0,1]\rightarrow\RR^n$ such that $\Haus^t(E\setminus \bigcup_i \Gamma_i)=0$. \end{theorem}

\begin{example}[$2^n$-corner Cantor sets] Let $E_t\subseteq[0,1]^n$ be the self-similar $2^n$-corner Cantor set of Hausdorff dimension $0<t<n$, which is obtained via similarities that dilate the unit cube by a factor $q=2^{-n/t}$ (see Figure \ref{fig:snow} for the case $n=2$). \begin{itemize}
\item If $1\leq m\leq n-1$ and $m\leq t<m+1$, then $E_t$ is purely $(\Haus^t,m\rightarrow t)$ unrectifiable by Theorem \ref{t:mm-pu}.
\item If $1\leq m\leq n-1$ and $m\leq t<m+1$, then $E_t$ is countably $(\Haus^t,m\rightarrow s)$ rectifiable for all $s>t$ by Theorem \ref{t:HBC}.
\item If $t<1$, then $E_t$ is countably $(\Haus^t,1)$ bi-Lipschitz rectifiable by Theorem \ref{t:HBC}.
\end{itemize}
\end{example}

\subsection{Main results and organization of the paper}

Let $1\leq m\leq n-1$ be integers and let $s\in[m,n]$. Recall that we define a \emph{$(m/s)$-H\"older $m$-cube} to be the image of a H\"older continuous map $f:[0,1]^m\rightarrow\RR^n$ with exponent $(m/s)$. When $m=1$, we refer to a $(1/s)$-H\"older $1$-cube as a \emph{$(1/s)$-H\"older curve}. By Proposition \ref{p:decomp}, every Radon measure (that is, a locally finite Borel regular outer measure) $\mu$ on $\RR^n$ can be written uniquely as $$\mu=\mu_{m\rightarrow s} + \mu_{m\rightarrow s}^\perp,$$ where \begin{itemize}
\item $\mu_{m\rightarrow s}$ is \emph{carried by $(m/s)$-H\"older $m$-cubes} in the sense that there exist countably many $(m/s)$-H\"older $m$-cubes $\Gamma_i\subseteq\RR^n$ such that $\mu_{m\rightarrow s}(\RR^n\setminus \bigcup_i \Gamma_i)=0$, and
\item $\mu_{m\rightarrow s}^\perp$ is \emph{singular to $(m/s)$-H\"older $m$-cubes} in the sense that $\mu(\Gamma)=0$ for every $(m/s)$-H\"older $m$-cube $\Gamma\subseteq\RR^n$.
\end{itemize} When $m=s=1$, Badger and Schul \cite[Theorem A]{BS3} gave a full characterization of the 1-rectifiable $\mu_{1\rightarrow 1}$ and purely 1-unrectifiable $\mu_{1\rightarrow 1}^\perp$ parts of a Radon measure. When $s=n$, the existence of space-filling curves implies that  for every Radon measure $\mu$ on $\RR^n$, $\mu=\mu_{m\rightarrow n}$ for all $1\leq m\leq n-1$. No other case of the following fundamental problem in geometric measure theory has been completely solved. When $s=m$ is an integer \emph{and the measure a priori satisfies $\mu\ll\Haus^m$}, several solutions to Problem \ref{p:ident} have been given by Federer \cite{Fed47}, Preiss \cite{Preiss}, Azzam and Tolsa \cite{AT15}, and Tolsa and Toro \cite{TT-rect}.

\begin{problem}[identification problem] \label{p:ident} Let $1\leq m\leq n-1$ be integers and let $s\in[m,n]$. Find geometric or measure-theoretic characterizations of being carried by or singular to $(m/s)$-H\"older $m$-cubes that identify  $\mu_{m\rightarrow s}$ and $\mu_{m\rightarrow s}^\perp$ for every Radon measures $\mu$ on $\RR^n$.\end{problem}

In our principal result, we show that extreme behavior of the lower $s$-density is sufficient to detect that a measure is carried by or singular to $(1/s)$-H\"older curves. In the statement, $\mu\res E$ denotes the restriction of the measure $\mu$ to the set $E\subseteq\RR^n$, defined by the rule $$\mu\res E(F)=\mu(E\cap F)\quad\text{for all $F\subseteq\RR^n$}.$$

\begin{maintheorem}[Behavior at extreme lower densities] \label{thm:A} Let $\mu$ be a Radon measure on $\RR^n$ and let $s\in[1,n)$. Then  the measure \begin{equation}\underline{\mu}^s_{\,0}:=\mu\res\left\{x\in\RR^n: \liminf_{r\downarrow 0} \frac{\mu(B(x,r))}{r^s}=0\right\}\end{equation} is singular to $(1/s)$-H\"older curves. At the other extreme, the measure \begin{equation}\label{eq:mu-perp}\underline{\mu}^s_{\,\infty}:=\mu\res\left\{x\in\RR^n: \int_0^1 \frac{r^s}{\mu(B(x,r))}\,\frac{dr}{r}<\infty\text{ and }\limsup_{r\downarrow 0} \frac{\mu(B(x,2r))}{\mu(B(x,r))}<\infty\right\}\end{equation} is carried by $(1/s)$-H\"older curves.\end{maintheorem}

\begin{remark} Note that Theorem \ref{thm:A} extends Theorem \ref{t:mm-den} to arbitrary Radon measures; see Corollary \ref{l:LD}. In \eqref{eq:mu-perp}, if the Dini-type condition $\int_0^1 [r^{s}/\mu(B(x,r))]\, r^{-1}dr<\infty$ holds at some $x$, then $\lim_{r\downarrow 0}r^{-s}\mu(B(x,r))=\infty$. Thus,  Theorem A identifies diverse behaviors of a measure on the points of vanishing lower density and the points of ``rapidly" infinite density. It is possible to remove the doubling condition in \eqref{eq:mu-perp} by using dyadic cubes; see Theorem \ref{t:fast}. The special case $s=1$ of Theorem A first appeared in \cite{BS1,BS2}.\end{remark}

The following corollary of Theorem \ref{thm:A} is immediate.

\begin{maincor} \label{c:B} Let $\mu$ be a Radon measure on $\RR^n$, let $s\in[1,n)$, and let $t\in[0,s)$. Then the measure \begin{equation}\mu^t_{+}:=\mu\res\left\{x\in\RR^n: 0<\liminf_{r\downarrow 0} \frac{\mu(B(x,r))}{r^t} \leq \limsup_{r\downarrow 0} \frac{\mu(B(x,r))}{r^t}<\infty\right\}\end{equation} is carried by $(1/s)$-H\"older curves.\end{maincor}

Our next pair of results gives sharpened versions of Corollary \ref{c:B}, depending on the values of $s$ and $t$. Recall that a set $\Gamma\subseteq\RR^n$ is a \emph{bi-Lipschitz curve} provided that $\Gamma=f([0,1])$ for some map $f:[0,1]\rightarrow\RR^n$ such that $$ L^{-1}|x-y|\leq |f(x)-f(y)|\leq L|x-y|\quad\text{for all }x,y\in[0,1],\text{ for some }1\leq L<\infty.$$ In line with the terminology above, we say that a measure $\mu$ is \emph{carried by bi-Lipschitz curves} if there exist countably many bi-Lipschitz curves $\Gamma_i$ such that $\mu(\RR^n\setminus\bigcup_{i} \Gamma_i)=0$. Theorem \ref{t:HBC} is an immediate consequence of Theorems \ref{thm:B} and \ref{thm:C} and the basic fact that $\limsup_{r\downarrow 0} r^{-t}\Haus^t(E\cap B(x,r))\leq C(t)<\infty$ at $\Haus^t$-a.e.~$x$, for every $t$-set $E$ (e.g. see \cite{Mattila}).

\begin{maintheorem}[Improvement to bi-Lipschitz curves] \label{thm:B} Let $\mu$ be a Radon measure on $\RR^n$ and let $t\in[0,1)$. Then the measure $\mu^t_{+}$ is carried by bi-Lipschitz curves.
\end{maintheorem}

\begin{maintheorem}[Improvement to $(m/s)$-H\"older $m$-cubes] \label{thm:C} Let $\mu$ be a Radon measure on $\RR^n$, let  $1\leq m\leq n-1$ be an integer, let $s\in[m,n)$, and let $t\in[0,s)$. Then the measure $\mu^t_{+}$ is carried by $(m/s)$-H\"older $m$-cubes.\end{maintheorem}

\begin{example} Garnett, Killip, and Schul \cite{GKS} examined a family of Radon measures $\{\mu_\delta:0\leq \delta\leq 1/3\}$ on $\RR^n$ ($n\geq 2$) such that $$0<\mu_\delta(B(x,2r))\leq C_\delta\,\mu_\delta(B(x,r))<\infty\quad\text{for all }x\in\RR^n\text{ and }r>0\qquad(0<\delta\leq 1/3),$$ where $\mu_0$ is a discrete measure and $\mu_{1/3}$ is Lebesgue measure. They proved that there exists $\delta_0=\delta_0(n)\in(0,1/3)$ such that for all $0<\delta\leq \delta_0$, the measure $\mu_\delta$ is simultaneously carried by Lipschitz curves \emph{and} singular to bi-Lipschitz curves (see the introduction of either \cite{BS1} or \cite{BS2}). As a consequence, Badger and Schul (see \cite[Example 1.15]{BS1}) showed that for all $0<\delta\leq \delta_0$, $$\lim_{r\downarrow 0} \frac{\mu_\delta(B(x,r))}{r}=\infty\quad\text{at $\mu_\delta$-a.e. $x\in\RR^n$}.$$ By Theorem \ref{thm:B}, we may conclude in addition that for all $0<\delta\leq \delta_0$ and for all $0<t<1$,  $$
\liminf_{r\downarrow 0} \frac{\mu_\delta(B(x,r))}{r^t}\in\{0,\infty\}\,\text{ or }\,\limsup_{r\downarrow 0} \frac{\mu_\delta(B(x,r))}{r^t} \in \{0,\infty\}\quad \text{at $\mu_\delta$-a.e.~$x\in\RR^n$}.$$
\end{example}

The remainder of the paper is organized into two parts. In Part I (\S\S 2--5), we develop several H\"older and bi-Lipschitz parameterization theorems for a variety of ``small'' sets, which are of separate interest (see especially Theorems \ref{thm:param} and \ref{thm:bilip}). In Part II (\S\S6--7), we derive Theorems \ref{thm:A}, \ref{thm:B}, and \ref{thm:C} using geometric measure theory techniques combined with the technology of Part I. Because it is focused on metric geometry of Euclidean sets, Part I may be read independently from the introduction and Part II.

\part{Parameterizations}

In this part of the paper, we develop several parameterization theorems, which identify certain ``small'' sets as subsets of ``regular" curves or surfaces. In \S2, we give a rather simple criterion for the leaves of a tree of sets to be contained in a $(1/s)$-H\"older curve. This result (see Theorem \ref{t:diam^s}) extends the special case $s=1$ (Lipschitz curves), which was described by Badger and Schul in \cite[\S3]{BS2}. In \S3--\S5, we prove that a Euclidean set with Assouad dimension $t$ strictly less than $s$ is always contained in a $(m/s)$-H\"older $m$-plane, where $m=\lfloor s\rfloor$. In other words, sets with small Assouad dimension are contained in H\"older surfaces (see Theorem \ref{thm:param}). In addition, we prove under an \emph{a priori} quantitative topological assumption that sets with small Assouad dimension are in fact contained in bi-Lipschitz surfaces (see Theorem \ref{thm:bilip}). The proof of this pair of results  incorporates and builds on ideas from MacManus' construction of quasicircles in \cite{MM}.

\section{Drawing H\"older curves through the leaves of summable trees}

\begin{definition}\label{def:tree} We define a \emph{tree of sets} $\mathcal{T}=\bigsqcup_{k=0}^\infty\mathcal{T}_k$ to be a nonempty collection of bounded sets in $\RR^n$ with \begin{enumerate}
\item[(\emph{i})] unique root: $\#\mathcal{T}_0=1$,
\item[(\emph{ii})] parents: for all $k\geq 1$ and $E\in\mathcal{T}_k$, there is an associated set $E^\uparrow\in\mathcal{T}_{k-1}$ called the \emph{parent} of $E$,
\item[(\emph{iii})] geometric diameters: there exist constants $0<\rho<1$ and $P_1,P_2>0$ such that $$P_1\rho^k \leq \diam E \leq P_2\rho^k$$ for all $k\geq 0$ and $E\in\mathcal{T}_k$, and
\item[(\emph{iv})] gap-diameter bound: there exists $P_3>0$ such that $$\gap(E,E^\uparrow):=\inf_{x\in E}\inf_{y\in E^\uparrow}|x-y|\leq P_3\diam E$$ for all $k\geq 1$ and for all $E\in\mathcal{T}_k$.
\end{enumerate} Let $\Top(\mathcal{T})$ denote the unique set in $\mathcal{T}_0$, which we call the \emph{top} of $\mathcal{T}$. An \emph{infinite branch} of $\mathcal{T}$ is a sequence $(E_k)_{k=0}^\infty$ in $\mathcal{T}$ with $E_0=\Top(\mathcal{T})$ and $E_k^\uparrow=E_{k-1}$ for all $k\geq 1$. A point $x\in\RR^n$ is called a \emph{leaf} of $\mathcal{T}$ if there exists an infinite branch $(E_k)_{k=0}^\infty$ and a sequence $(x_k)_{k=1}^\infty$ with $x_k\in E_k$ for all $k\geq 0$ such that $x=\lim_{k\rightarrow\infty} x_k$. We let $$\leaves(\mathcal{T})=\{x\in\RR^n:x\emph{ is a leaf of }\mathcal{T}\}$$ denote the set of \emph{leaves} of $\mathcal{T}$.
\end{definition}

\begin{remark}\label{rk:tree} Definition \ref{def:tree} is loosely modeled on a tree of dyadic cubes, but designed with additional flexibility for applications (e.g.~see Theorem \ref{t:Ss}). Axioms $(\emph{iii})$ and $(\emph{iv})$ in the definition ensure that every infinite branch admits a unique leaf of $\mathcal{T}$: \emph{For every infinite branch $(E_k)_{k=0}^\infty$ of $\mathcal{T}$, there exists $x\in\leaves(\mathcal{T})$ such that $x=\lim_{k\rightarrow\infty} x_k$ for every sequence $(x_k)_{k=0}^\infty$ with $x_k\in E_k$ for all $k\geq 0$.}
\end{remark}

In the main result of this section, we prove that the leaves of a tree of sets with summable diameters are contained in a H\"older curve. The special case of Lipschitz curves (see the proof of \cite[Lemma 3.3]{BS2}) is easier, because of the special fact that every connected, compact set in $\RR^n$ of finite $\Haus^1$ measure is necessarily the image of a Lipschitz map $f:[0,1]\rightarrow\RR^n$ (see \cite[Theorem 4.4]{AO-curves}). For higher dimensional curves, we must construct the H\"older parameterization by hand.

\begin{theorem} \label{t:diam^s} Let $\mathcal{T}$ be a tree of sets in the sense of Definition \ref{def:tree}. If $s\geq 1$ and \begin{equation}S^s(\mathcal{T}):=\label{e:diam^s}\sum_{E\in\mathcal{T}} (\diam E)^s<\infty,\end{equation} then $\Haus^s(\leaves(\mathcal{T}))=0$ and there exists a $(1/s)$-H\"older map $f:[0,1]\rightarrow\RR^n$ such that $\leaves(\mathcal{T})\subseteq f([0,1])$. Moreover, the $(1/s)$-H\"older constant of $f$ can be taken to depend only on $S^s(\mathcal{T})$ and the geometric parameters of the tree ($\rho,P_1,P_2,P_3$).\end{theorem}

\begin{proof} Assume that \eqref{e:diam^s} holds for some $s\geq 1$. Replacing each set in $\mathcal{T}$ with its closure, we may assume without loss of generality that the sets in $\mathcal{T}$ are closed. By deleting sets from $\mathcal{T}$ if necessary, we may also assume without loss of generality that every set in $\mathcal{T}$ belongs to an infinite branch. For all $k\geq 0$ and for all $E\in\mathcal{T}_k$, choose a point $x_{k,E}\in E$. Construct a connected set $\Gamma_\circ$ by drawing a line segment from $x_{k,E}$ to $x_{k-1,E^\uparrow}$ for all $k\geq 1$ and $E\in\mathcal{T}_k$, and let $\Gamma$ denote the closure of $\Gamma_\circ$. By Remark \ref{rk:tree}, $\Gamma$ contains  $\leaves(\mathcal{T})$. Our present goal is to show that $\Gamma$ admits a $(1/s)$-H\"older parameterization by a closed interval. More specifically, for each $k\geq 0$ and $E\in\mathcal{T}_k$, define $$M_{k,E}:=2 \sum_{j={k+1}}^\infty\sum_{\stackrel{F\in\mathcal{T}_j}{F^{\uparrow\dots}=E}} |x_{j,F}-x_{j-1,F^\uparrow}|^s,$$ where the sum is over all descendants of $E$ in $\mathcal{T}$. Note that $M:=M_{0,\Top(\mathcal{T})}<\infty$, because: \begin{equation*}
\begin{split}\frac{1}{2}M_{0,\Top(\mathcal{T})} &\leq \sum_{j=1}^\infty\sum_{F\in\mathcal{T}_j}(\diam F+\gap(F,F^\uparrow)+\diam F^\uparrow)^s \\ &\leq \sum_{j=1}^\infty\sum_{F\in\mathcal{T}_j}\left(\left(1+P_3+\frac{P_2}{P_1\rho}\right)\diam F\right)^s<\left(1+P_3+\frac{P_2}{P_1\rho}\right)^s S^s(\mathcal{T})<\infty\end{split}\end{equation*} by \eqref{e:diam^s}, where $P_1$, $P_2$, $P_3$, and $\rho$ are the geometric parameters of $\mathcal{T}$ (see Definition \ref{def:tree}). We will construct a $(1/s)$-H\"older continuous map $g:[0,M]\rightarrow\RR^n$ such that $\Gamma=g([0,M])$ by defining a sequence $g_k:[0,M]\rightarrow\RR^n$ of piecewise linear maps whose limit is $g$.

For each $k\geq 1$, the image of $g_k$ is the ``truncated tree" $\Gamma_k:=\bigcup_{j=1}^k\bigcup_{E\in\mathcal{T}_j}[x_{j,E},x_{j-1,E^\uparrow}]$. Roughly speaking, $g_k$ is defined by starting at the root $x_{0,\Top(\mathcal{T})}$ and then touring each of the edges $[x_{j,E},x_{j-1,E^\uparrow}]$ in $\Gamma_k$ twice (once ``down the tree", once ``up the tree") at speed $$\left|x_{j,E}-x_{j-1,E^\uparrow}\right|^{1-s},$$ with the additional rule that whenever reaching a point $x_{k,E}$ in the ``lowest level" of $\Gamma_k$, we pause for $M_{k,E}$ time units in the domain of $g_k$ before continuing the tour. By defining the maps $g_1, g_2,\dots$ inductively, one can ensure that $g_k(t)=g_{k+1}(t)$ for all times $t$ where $g_k'(t)\neq 0$ (that is, for all times where the tour of $\Gamma_k$ is not paused). Or, in other words, one can construct $g_{k+1}$ from $g_k$ by  modifying the definition of $g_k$ only in the intervals where the tour of $\Gamma_k$ was paused. We leave it to the reader to give a precise definition of the maps $g_k$ if desired. The salient facts of any such construction are these: \begin{itemize}
\item $g_k([0,M])=\Gamma_k$ for all $k\geq 1$;
\item $|g_k(x)-g_k(y)| \leq A_k|x-y|$ for all $k\geq 1$ and $x,y\in[0,M]$, where $$A_k :=\sup_{1\leq j\leq k}\sup_{E\in \mathcal{T}_j} \left|x_{j,E}-x_{j-1,E^\uparrow}\right|^{1-s} \leq \left(P_2(1+P_3+\rho^{-1})\right)^{1-s}\rho^{k(1-s)};\text{ and,}$$
\item $|g_k(x)-g_{k+1}(x)| \leq B_k$ for all $k\geq 1$ and $x\in[0,M]$, where $$B_k:=\sup_{E\in\mathcal{T}_{k+1}}|x_{k+1,E}-x_{k,E^\uparrow}|\leq P_2(\rho+P_3\rho+1)\rho^k.$$
\end{itemize} Define $g:[0,M]\rightarrow \R^n$ pointwise by $g(t):=\lim_{k\rightarrow\infty} g_k(t)=g_1(t)+\sum_{k=1}^\infty (g_{k+1}(t)-g_k(t))$. The existence and continuity of $g$ are immediate, since $\sum_{k=1}^\infty B_k<\infty$. From this point, it is a standard exercise to show that $g$ is $(1/s)$-H\"older continuous with H\"older constant depending on at most $P_1$, $P_2$, $P_3$, $\rho$, and $M$ (hence on $S^s(\mathcal{T})$); cf.~\cite[Lemma VII.2.8]{ss-reals}. It is also easy to check that $g([0,M])=\Gamma$.

It remains to show that $\Haus^s(\leaves(\mathcal{T}))=0$. If $E,E_1\in\mathcal{T}$ and $E_1$ is a descendent of $E$ (i.e.~$E_1^\uparrow=E$), then $$\sup_{x\in E_1}\dist(x,E) \leq \gap(E_1,E)+\diam E_1 \leq (P_3+1)\diam E_1 \leq (P_3+1)\,\frac{P_2}{P_1}\rho\diam E.$$ More generally, if $E,E_k\in\mathcal{T}$ and $E_k$ is a $k$th descendent of $E$ (i.e.~ $E_k^\uparrow =E_{k-1}$, \dots, $E_2^\uparrow = E_1$, and $E_1^\uparrow = E$), then \begin{equation*}\begin{split}\sup_{x\in E_k}\dist(x,E) &\leq \sup_{x\in E_k}\dist(x,E_{k-1})+\dots+\sup_{x\in E_1}\dist(x, E)\\ &\leq (P_3+1)\,\frac{P_2}{P_1}(\rho^k+\dots+\rho)\diam E.\end{split}\end{equation*} Thus, for all $E\in\mathcal{T}$, the set $$\widetilde{E}:=\left\{x\in\RR^n: \dist(x, E) \leq  (P_3+1)\frac{P_2}{P_1}\frac{\rho}{1-\rho}\diam E\right\}$$ contains all descendants of $E$. Hence $\leaves(\mathcal{T}) \subseteq \bigcup_{T\in \mathcal{T}_k} \widetilde{T}$ for all $k\geq 1$. In particular, we can bound the $s$-dimensional Hausdorff content $$\Haus^s_\infty(\leaves(\mathcal{T})) \leq \sum_{E\in\mathcal{T}_k}(\diam\widetilde{E})^s\leq C(P_1,P_2,P_3,\rho) \sum_{E\in\mathcal{T}_k} (\diam E)^s\quad\text{for all $k\geq 1$}.$$ Letting $k\rightarrow\infty$, we obtain $\Haus^s(\leaves(\mathcal{T}))=0$, because $\sum_{E\in\mathcal{T}}(\diam E)^s<\infty$.
\end{proof}

\section{Drawing surfaces through sets with small Assouad dimension}
\label{s:surfaces}

In this section, we present four related parameterization theorems, which draw surfaces through ``small" and/or uniformly disconnected sets; see Theorems \ref{thm:param}, \ref{thm:bilip}, \ref{thm:biHolder}, and \ref{thm:QS}. The notion of size that we use for this purpose is Assouad dimension; for an in depth survey of this concept, see \cite{Luukk}.

\begin{definition}[Assouad dimension] \label{def:assouad} Let $X$ be a metric space, let $\beta>0$, and let $C>1$. We say $X$ is \emph{$(C,\beta)$-homogeneous} (or simply, $X$ is \emph{$\beta$-homogeneous}) if for every bounded set $A\subseteq X$ and for every $\d\in(0,1)$, there exist sets $A_1,\dots,A_N \subseteq X$ such that $$ A_1\cup\dots\cup A_N\supseteq A,\quad \diam A_i \leq \delta \diam A, \quad\text{and}\quad\card\{A_1,\dots,A_N\}\leq  C \d^{-\beta}.$$ The \emph{Assouad dimension} of $X$, denoted $\dim_A X$, is defined by
\[ \dim_A X  := \inf\{\beta>0 : X\text{ is }\beta\text{-homogeneous}\}\in[0,\infty]. \] \end{definition}

The first parameterization theorem extends \cite[Theorem 3.4]{MM2000}, a measure-theoretic condition for an $s$-set to be contained in a H\"older surface, which is stated  without a proof. For the connection between Assouad dimension and the condition in \cite{MM2000}, see Lemma \ref{lem:1} below.
We supply a proof of Theorem \ref{thm:param} in \S\ref{sec:holder}.

\begin{theorem}[H\"older parameterization] \label{thm:param}
Let $1\leq m\leq n-1$ be integers, let $s\in[m,n)$, and let $E\subseteq\RR^n$. If $\dim_AE<s$, then there exists an $(m/s)$-H\"older continuous map $f:\R^m \to \R^n$ such that $E\subseteq f(\R^m)$.
\end{theorem}

In order to obtain bi-Lipschitz parameterizations or even bi-H\"older parameterizations, it is natural to impose topological assumptions on the set. To wit, it is well known that if $E \subseteq \R^n$ is a totally disconnected closed set, then $E$ is homeomorphic to a subset of the standard ternary Cantor set (e.g.~see \cite[Theorem 7.8]{Kechris}). Furthermore, for every positive integer $m<n$, there exists a topological embedding $f:\R^m \to \R^n$ whose image contains $E$ (e.g.~see \cite{Rushing}). To ensure that $f$ satisfies additional regularity properties, we employ a scale-invariant version of total disconnectedness, which was introduced by David and Semmes \cite[\S15]{DS}.

\begin{definition}[uniformly disconnected spaces] Let $X$ be a metric space. We say that $X$ is \emph{$c$-uniformly disconnected} for some $c\geq 1$ if for all $x\in X$ and for all $0<r\leq\diam{X}$, there exists $E_{x,r}\subseteq X$ containing $x$ such that $\diam{E_{x,r}} \leq r$ and $\gap(E_{x,r},X\setminus E_{x,r}) \geq r/c$. To suppress dependence on $c$, we may simply say that $X$ is \emph{uniformly disconnected}.\end{definition}

The second parameterization theorem states that under the additional assumption of uniform disconnectedness, the H\"older parameterization in Theorem \ref{thm:param} upgrades to a bi-Lipschitz parametrization. We provide a proof of Theorem \ref{thm:bilip} in \S\ref{sec:BLthm}.

\begin{theorem}[bi-Lipschitz parameterization] \label{thm:bilip} Let $1\leq m\leq n-1$ be integers. If $E\subseteq\RR^n$ is uniformly disconnected and $\dim_A E<m$, then  there exists a bi-Lipschitz embedding $f:\R^m \to \R^n$ such that $E\subseteq f(\R^m)$.\end{theorem}

When $\dim_A E<1$, the set $E$ is uniformly disconnected by \cite[Proposition 5.1.7]{McTy} (also see \cite[Lemma 15.2]{DS}). Thus, Euclidean sets of Assouad dimension strictly less than one are always contained in a bi-Lipschitz line.

\begin{corollary} \label{cor:bilip} Let $E\subseteq\RR^n$. If $\dim_A E<1$, then there exists a bi-Lipschitz embedding $f:\RR\rightarrow\RR^n$ such that $E\subseteq f(\RR)$.\end{corollary}

\begin{remark} The upper bound on the Assouad dimension in Theorem \ref{thm:bilip} is necessary in that there do not exist bi-Lipschitz embeddings $f:X\to \R^n$ when $\dim_A X\geq n$. This assertion follows from the well-known fact that Assouad dimension is a bi-Lipschitz invariant (see \cite[Theorem A.5]{Luukk}) and the fact that uniformly disconnected sets in $\R^n$ are porous and have Assouad dimension strictly less than $n$ (see \cite[Theorem 5.2]{Luukk}).\end{remark}

In the event that $E$ is uniformly disconnected and $\dim_A E \geq m$, one may expect that the bi-Lipschitz embedding of Theorem \ref{thm:bilip} could be replaced by a $(1/\gamma)$-bi-H\"older embedding of $\R^m$ provided that $\dim_A E <\gamma m$. However, $(1/\gamma)$-bi-H\"older embeddability of $\R^m$ into $\R^n$ when $\gamma m < n$ is a formidable problem, which has been solved only in special cases such as when $m=1$ \cite{BoHei, RV},  when $\frac{1}{\gamma}$ is sufficiently close to $1$ \cite{DT99,DT}, or when $n$ is much bigger than $\gamma m$ \cite{Assouad2}. By modifying the proof of Theorem \ref{thm:bilip} using the snowflaking techniques for polygonal paths from \cite{BoHei, RV}, one can obtain the following bi-H\"older variant of the bi-Lipschitz parameterization theorem when $m=1$. Because we do not need Theorem \ref{thm:biHolder} for our applications in Part II, we leave details of the proof of the theorem to the interested reader.

\begin{theorem}[bi-H\"older parameterization] \label{thm:biHolder} Let $n\geq 2$ be an integer and let $s\in[1,n)$. If $E\subseteq\RR^n$ is uniformly disconnected and $\dim_A E<s$, then there exists a $(1/s)$-bi-H\"older embedding $f:\RR\rightarrow\RR^n$ (that is, both $f$ and $f^{-1}$ are $(1/s)$-H\"older) such that $E\subseteq f(\R)$.
\end{theorem}

MacManus \cite{MM} proved that if $E\subseteq \R^n$ is uniformly disconnected, then there exists a quasisymmetric\footnote{For the definition of and background on quasisymmetric maps, we refer the reader to \cite{Heinonen}.} embedding $f:\R \to \R^n$ whose image contains $E$. Note that this result does not require a bound on the dimension of $E$, which is natural in view of the fact that quasisymmetric maps may increase the dimension of sets. Arguments similar to ones used in the proof of Theorem \ref{thm:bilip} (also see the proof of \cite[Theorem 6.3]{Vais1}) can be used to obtain the following extension of MacManus' result. We leave details of the proof of Theorem \ref{thm:QS} to the interested reader; see Remark \ref{rem:QS}.

\begin{theorem}[quasisymmetric parameterization]\label{thm:QS}
Let $1\leq m\leq n-1$ be integers. If $E\subseteq\RR^n$ is uniformly disconnected, then there exists a quasisymmetric embedding $f:\R^m \to \R^n$ such that $E \subseteq f(\R^m)$.
\end{theorem}

\begin{remark}
It is natural to ask to what extent do Theorems \ref{thm:param}, \ref{thm:bilip}, \ref{thm:biHolder}, and \ref{thm:QS} hold with other metric spaces. We leave this question open for future research.
\end{remark}

In the remainder of this section, we fix some basic notation used in our construction of the surfaces appearing in Theorems \ref{thm:param} and \ref{thm:bilip}. Section \ref{sec:holder} (the proof of Theorem \ref{thm:param}) and section \ref{sec:BLthm} (the proof of Theorem \ref{thm:bilip}) may be read independently of each other.

\subsection*{Cubes}\label{sec:thick}

Given $l\in(0,\infty)$ and integers $0\leq k \leq n$, a \emph{$k$-cube} of \emph{side length $l$}, $$\mathcal{K}:=I_1\times\cdots\times I_n\subseteq\RR^n,$$ is a product of bounded, closed intervals $I_i\subseteq\R$ such that the length $|I_i| = l$ for $k$ indices and $|I_i| = 0$ for $n-k$ indices. If $\mathcal{K}$ is an $n$-cube, then the topological boundary $\partial\mathcal{K}$ of $\mathcal{K}$ is the union of $2n$-many $(n-1)$-cubes, which are called the \emph{$(n-1)$-faces} of $\mathcal{K}$. In general, for all $1\leq k \leq n$, every $k$-face of $\mathcal{K}$ is the union of $2k$-many $(k-1)$-cubes, which are called the \emph{$(k-1)$-faces} of $\mathcal{K}$. The $0$-faces and $1$-faces of a cube $\mathcal{K}$ are commonly called the \emph{vertices} ane \emph{edges}, respectively, of $\mathcal{K}$. For each $(n-1)$-face $F$ of an $n$-cube $\mathcal{K}$, there is a unique face $\tilde{F}$ of $\mathcal{K}$ such that $$F\cap \tilde{F} = \emptyset;$$ we call $\tilde{F}$ the \emph{antipodal face} of $F$.

For all $x=(x^1,\dots,x^n)\in\RR^n$ and $r>0$, let $\mathcal{C}^n(x;r)\subseteq\RR^n$ denote the $n$-cube centered at $x$ of side length $2r>0$ with edges parallel to the coordinate axes; that is,
\[ \mathcal{C}^n(x;r) := [x^1-r,x^1+r]\times\cdots\times[x^{n}-r,x^{n}+r].\]
For all $x\in\RR^n$ and $0<r<R$, let $\mathcal{A}^n(x;r,R)\subseteq\RR^n$ denote the closed annular region between $\mathcal{C}^n(x;r)$ and $\mathcal{C}^n(x;R)$,
\[ \mathcal{A}^n(x;r,R) := \overline{\mathcal{C}^n(x;R)\setminus \mathcal{C}^n(x;r)}.\]

\subsection*{Grids}

For all $\d>0$, let $\mathscr{G}_{\d}$ denote the grid of cubes of side length $\delta$,
\[ \mathscr{G}_{\d} := \{ [m_1 \d,(m_1+1)\d] \times \cdots \times [m_n \d,(m_n+1)\d] : m_1,\dots,m_n\in\mathbb{Z}\}.\]
For all integers $0\leq k<n$, define the \emph{$k$-skeleton} of $\mathscr{G}_\delta$,
\[ \mathscr{G}_{\d}^k := \bigcup\left\{F:F\text{ is a $k$-face of some }\mathcal{K}\in\mathscr{G}_\delta\right\}. \]
For instance, $\mathscr{G}^0_{\d}$ is the set of all vertices of cubes in $\mathscr{G}_\delta$, while $\mathscr{G}^1_{\d}$ is the union of all edges of cubes in $\mathscr{G}_\delta$.

\begin{remark}[length bound]\label{rem:length}
Suppose that $M$ is a union of cubes in $\mathscr{G}_d$ with $\diam{M}=D\delta$ and $\g \subseteq \mathscr{G}_{\d}^1$ is an arc with endpoints in $\mathscr{G}_\delta^0$ such that $\g \subseteq M$. Then $M$ is formed from at most $D^n$ distinct cubes in $\mathscr{G}_{\d}$, each of which has $n 2^{n-1}$-many edges. Thus, since an arc contained in $\mathscr{G}^1_\delta$ traverses each edge at most once, \begin{equation*}\Haus^1(\g) \leq D^n2^{n-1}n\delta=C(n)\left[\delta^{-1}\diam M\right]^n\delta.\end{equation*}
\end{remark}

\subsection*{Tubes}

Given $0<\e\leq \d$ and an oriented polygonal arc $\g\subseteq \mathscr{G}^1_{\d}$ with initial endpoint $y_1 \in \mathscr{G}^0_{\d}$ and terminal endpoint $y_2\in \mathscr{G}^0_{\d}$, define
\[ \tube_{\varepsilon}(\g) = \bigcup\left\{ \mathcal{C}^n(x;\e/2):x\in \g\text{ and }\dist(x,\{y_1,y_2\})\geq \e/2\right\},\] the \emph{tube around $\gamma$} of \emph{width $\varepsilon$},
where the union is taken over all points $x\in \g$ whose distance from the endpoints of $\g$ are at least $\e/2$. Distinguishing between the initial endpoint $y_1$ and terminal endpoint $y_2$ of $\gamma$, we can split the topological boundary $\partial T$ of $T = \tube_{\e}(\g)$ into three distinguished pieces:
\begin{itemize}
\item $\entrance(T)$ is the $(n-1)$-cube in $\partial T$ of side length $\varepsilon$ that contains $y_1$;
\item $\exit(T)$ is the $(n-1)$-cube in $\partial T$ of side length $\varepsilon$ that contains $y_2$; and,
\item $\side(T)$ is closure of the remainder of $\partial T$, i.e. $$\side(T) := \overline{\partial T \setminus (\entrance(T) \cup \exit(T))}.$$
\end{itemize}

\begin{lemma}[straightening tubes]\label{lem:blthick}
Let $0<\varepsilon\leq\delta/(8\sqrt{n})$, let $\gamma \subseteq \mathscr{G}_\delta^1$ be a simple oriented polygonal arc with distinct endpoints in $\mathcal{G}_\delta^0$, and let $T$ denote $\tube_\varepsilon(\gamma)$. Then there exists $L=L(n,\delta^{-1}\Haus^1(\gamma))>1$ and an $L$-bi-Lipschitz orientation preserving map $$\phi: T \to [0,\mathcal{H}^1(\g)]\times [-\e/2,\e/2]^{n-1}$$ such that the restrictions of $\phi$ to $\entrance(T)$ and $\exit(T)$ are isometries with $$\phi(\entrance(T)) = \{0\}\times[-\e/2,\e/2]^{n-1},$$ $$\phi(\exit(T)) = \{\Haus^1(\g)\}\times[-\e/2,\e/2]^{n-1}.$$
\end{lemma}

\begin{figure}
\begin{center}\includegraphics[width=.7\textwidth]{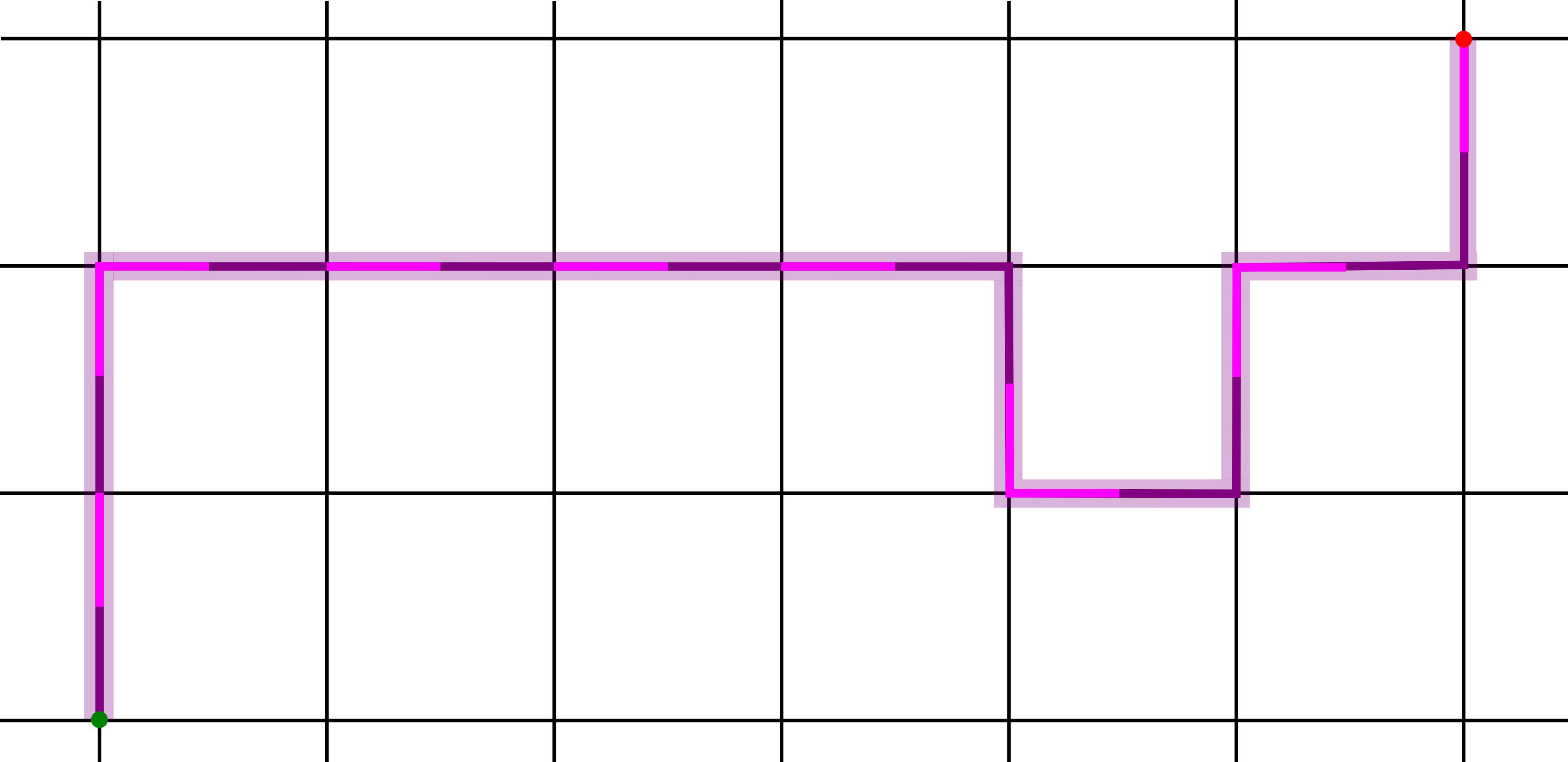}\end{center}
\caption{Proof of Lemma \ref{lem:blthick}}\label{fig:tube-proof}
\end{figure}

\begin{proof}
We start by breaking up $T$ into canonical tubes. Consider the two arcs
$${\G_1} := [0,\tfrac12]\times\{0\}^{n-1}\quad\text{and}\quad {\G_2} := ([0,\tfrac12]\times\{0\}^{n-1})\cup(\{\tfrac12\}\times [0,\tfrac12]\times\{0\}^{n-2}).$$
Then $\gamma$ can be written as a concatenation of consecutive arcs (see Figure \ref{fig:tube-proof}) $\gamma^1, \dots,\gamma^m$ such that for each $i=1,\dots,m$ either $\gamma^{i}$ is congruent to $\Gamma_{1}$ and $\mathcal{H}^1(\gamma^i) = \tfrac12\delta$ or $\gamma^{i}$ is congruent to $\Gamma_{2}$ and $\mathcal{H}^1(\gamma^i) = \d$. Thus, we can decompose $T$ as an almost disjoint union of consecutive tubes $$\tube_\varepsilon(\gamma^1),\quad\tube_\varepsilon(\gamma^2),\quad \dots,\quad \tube_{\varepsilon}(\gamma^{m-1}),\quad \tube_\varepsilon(\gamma^m),$$ intersecting in  $(n-1)$-cubes congruent to $\{0\}\times[-\varepsilon/2,\varepsilon/2]^{n-1}$ (see Figure \ref{fig:tube-proof}).

In view of the possible, simple geometric configurations, there exists $L_0=L_0(n)>1$ such that for each $\tube_\varepsilon(\gamma^j)$, there exists an $L_0$-bi-Lipschitz map
$$\phi_j:\tube_\varepsilon(\gamma^j)\rightarrow [\lambda_{j-1},\lambda_j]\times[-\varepsilon/2,\varepsilon/2]^{n-1},$$ where $\lambda_0=0$ and $\lambda_{j}-\lambda_{j-1}=\Haus^1(\gamma^j)$ for all $1\leq j\leq m$. Moreover, working sequentially, we are plainly free to choose the maps so that
\begin{itemize}
\item $\phi_j$ maps $\entrance(\tube_\varepsilon(\gamma^j))$ and $\exit(\tube_\varepsilon(\gamma^j))$ isometrically onto $$\{\lambda_{j-1}\}\times [-\varepsilon/2,\varepsilon/2]^{n-1}\quad\text{and}\quad \{\lambda_{j}\}\times [-\varepsilon/2,\varepsilon/2]^{n-1},$$ respectively, for all $1\leq j\leq m$, and
     \item $\phi_{j}|_{\entrance(\tube_\varepsilon(\gamma^{j}))}=\phi_{j-1}|_{\exit(\tube_\varepsilon(\gamma^{j-1}))}$ for all $2\leq j\leq m$.
    \end{itemize}
 Thus, we can define a map $$\phi:\tube_\varepsilon(\gamma)\rightarrow [0,\Haus^1(\gamma)]\times[-\varepsilon/2,\varepsilon/2]^{n-1}$$ by setting $\phi|_{\tube_\varepsilon(\gamma^j)}=\phi_j$ for all $1\leq j\leq m$.

To see that $\phi$ is bi-Lipschitz, let $x,y\in T$. If there exists $j\in\{2,\dots,m\}$ such that $x,y \in \tube_\varepsilon(\gamma^{j-1})\cup \tube_\varepsilon(\gamma^j)$, then---once again---in view of the possible, simple geometric configurations, the restriction of $\phi$ to $\tube_\varepsilon(\gamma^j)\cup \tube_\varepsilon(\gamma^j)$ is $L_1$-bi-Lipschitz for some $L_1 = L_1(n)>1$. Alternatively, if $x\in \tube_\varepsilon(\gamma^j)$ and $y \in \tube_\varepsilon(\gamma^i)$ for some $j>i+1$, then
 \[ |x-y| \geq \gap(\tube_\varepsilon(\gamma^j),\tube_\varepsilon(\gamma^i)) \geq \gap(\g^{i},\g^{j}) - \sqrt{n}\e \geq \delta/2 - \delta/4= \delta/4,\] because $\gamma$ has distinct endpoints and $0<\varepsilon\leq\delta/(8\sqrt{n})$.  Hence $$|\phi(x)-\phi(y)| \leq \sqrt{\Haus^1(\gamma)^2+n\varepsilon^2} < 1.5\Haus^1(\gamma)\leq \frac{6}{\delta}\Haus^1(\gamma)|x-y|,$$ again because $0<\varepsilon\leq\delta/(8\sqrt{n})$. A similar argument shows that $|\phi(x)-\phi(y)|\geq \delta/2$ and $$|x-y| \leq \frac{3}{\delta}\Haus^1(\gamma)|\phi(x)-\phi(y)|.$$
Therefore, $\phi$ is $L$-bi-Lipschitz for some $L=L(n, \d^{-1}\mathcal{H}^1(\g))$.
\end{proof}

\section{Proof of Theorem \ref{thm:param} (H\"older surfaces)}\label{sec:holder}

We first treat the case that $m=n-1$ of Theorem \ref{thm:param} in Proposition \ref{prop:n-1}. Afterwards, we derive the proof of Theorem \ref{thm:param} in general codimension.

\begin{proposition}[{H\"older parametrization in codimension one}]\label{prop:n-1}
For every choice of $n\geq 2$, $s\in[n-1,n)$, $C>1$, and $0\leq \beta <s$, there exist constants $$C'(n,s,C,\beta)>1\quad\text{and}\quad L = L(n,s,C,\beta)>1$$ such that for every $(C,\beta)$-homogeneous, closed set $E\subseteq\RR^n$, there exists a $(C',(n-1)\beta/s)$-homogeneous, closed set $E'\subseteq \R^{n-1}$ and a $(n-1)/s$-H\"older map $f:E' \to E$ such that $$\Hold_{(n-1)/s} f \leq L\quad\text{and}\quad f(E') = E.$$
\end{proposition}

\begin{proof}[{Proof of Proposition \ref{prop:n-1}}] Let us abbreviate $(n-1)/s=:\alpha$.
Since $E$ is $(C,\beta)$-homogeneous, there exists $C_0 = C_0(n,C,\beta)$ with the following property: \begin{quotation}\emph{For all $k\in\N$, $x\in\R^n$, and $r>0$, if the cube $Q:=\mathcal{C}^n(x;r)$ is divided into $k^n$ essentially disjoint subcubes $Q_1,\dots,Q_{k^n}$ of side length $2r/k$, then
\begin{equation*}\card{\{i\in\{1,\dots,k^n\} : Q_i\cap E \neq \emptyset\}} \leq C_0 k^{\beta}. \end{equation*}}\end{quotation}%
Fix a number $\e \in (0,1)$ so that $\e^{-1}$ is an integer and
\[ \e \leq (3^{n-1}C_0)^{\frac1{\beta-s}}. \] Finally, set $N=1+\lfloor C_0 \e^{-\beta} \rfloor$. We split the argument into two steps.

\bigskip\noindent\textsc{Step 1.} Assume that $E$ is bounded. Composing with appropriate similarities on the domain and target of $f$, we may assume that $E\subseteq [-1,1]^n$. Our goal is to construct a $(C',\beta\alpha)$-homogeneous, closed set $E'\subseteq [-1,1]^{n-1}$ and a $\alpha$-H\"older map $f:E' \to \R^n$ with $f(E') = E.$

Define $M_{\emptyset} := [-1,1]^n$. By way of induction, assume that a cube $M_w$ has been defined for some finite word $w$ with letters drawn from $\{1,\dots,N\}$. Divide $M_w$ into $\e^{-n}$-many subcubes with side lengths $2\e^{|w|+1}$ and mutually disjoint interiors; label the subcubes that intersect $E$ as $M_{w1},\dots,M_{wN_{w}}$. Note that $N_w\leq N$. Let $\mathcal{W}$ denote the set of all words, for which $M_w$ has been defined.

Define $M_{\emptyset}' = [-1,1]^{n-1}$. Inductively, given a cube $M_w' = \mathcal{C}^{n-1}(z_w;\e^{|w|/\alpha})$, the upper bound of $N_w$ and the choice of $\e$ allow us to find cubes $$M_{wi}' = \mathcal{C}^{n-1}(z_{wi};\e^{(|w|+1)/\alpha})\subseteq M_w'\quad(i\in\{1,\dots,N_w\})$$ such that for all distinct $i,j \in \{1,\dots,N_w\}$,
\[ \gap(M_{wi}',M_{wj}') \geq \e^{(|w|+1)/\alpha}.\]
Set $$E' := \bigcap_{k\in\N}\bigcup_{w\in\W, |w|=k}M'_w.$$ Then $E'$ is compact and $(C',\alpha\beta)$-homogeneous for some $C' = C'(n,s,C,\beta)$.

For each $x \in E'$, there exists a \emph{growing} sequence $(w_k)$ of words in $\W$ (i.e.~for each $k$, the word $w_{k+1}=w_ki$ for some $i\in\{1,\dots, N\}$) such that $$\bigcap_{k\in\N} M_{w_k}'=\{x\}.$$  Given a point $x\in E'$ and an associated sequence $(w_k)$, define $f(x)$ to be the unique point in $\bigcap_{k\in\N} M_{w_k}$. To show that $f$ is $\alpha$-H\"older continuous, fix $x,y\in E'$ and let $w$ denote the longest word such that $x,y\in M_{w}'$. Then
\[ |f(x)-f(y)| \leq \diam{M_w} = 2\sqrt{n}\e^{|w|} \leq \e^{-1} (\gap(M_{wi}',M_{wj}'))^{\alpha} \leq \e^{-1} |x-y|^{\alpha}.\] Thus, $f$ is $\alpha$-H\"older continuous and $f(E')\subseteq E$. In fact, because every point in $E$ can be represented as the unique point in $\bigcap_{k\in\N} M_{w_k}$ for some growing sequence $(w_k)$ of words in $\mathcal{W}$, we have $f(E')=E$.

\bigskip\noindent\textsc{Step 2.} Assume that $E$ is unbounded. For each $k\geq 0$, let $$E_k := E \cap [-\e^{-k},\e^{-k}]^n,$$ where $\e$ continues to denote the parameter chosen above. By adding the origin to the set $E$ if necessary, we may assume that $E_k\neq \emptyset$ for all $k\geq 0$. Each set $E_k$ is $(C,\beta)$-homogeneous, because $E$ is $(C,\beta)$-homogeneous and homogeneity is inherited by subsets.

By \textsc{Step 1}, there exists a $(C',\alpha\beta)$-homogeneous set $E_0' \subseteq \R^{n-1}$, a constant $L_0=L_0(n,s,C,\beta)>1$, and a $\alpha$-H\"older continuous map $f_0:E_0 \to \R^n$ such that $$\Hold_{\alpha} f \leq L_0\quad\text{and}\quad f_0(E_0') = E_0.$$

Inductively, suppose that for some $k\geq 0$, we have defined a $(C',\alpha\beta)$-homogeneous set $E_k' \subseteq [-\e^{-k/\alpha},\e^{-k/\alpha}]^{n-1}$ and a $\alpha$-H\"older map $f_k:E_k' \to \R^n$ such that $\Hold_{\alpha} f \leq L$ and $f_k(E_k') = E_k$. Divide $Q_{k+1} = [-\e^{-k-1},\e^{-k-1}]^n$ into $\e^{-n}$-many cubes with mutually disjoint interiors and side lengths $2\e^{-k}$ and denote by $\mathscr{Q}_{k+1}$ this collection of cubes. Let $Q_{k,1},\dots Q_{k,m_k}$ be those cubes in $\mathscr{Q}_{k+1}$ that intersect with $E$. Set $Q_{k,1} = [-\e^{-k},\e^{-k}]^n$. Since $m_k\leq N$, we can find, cubes $Q_{k,1}', \dots, Q_{k,m_k}'$ in $Q_{k+1}' = [-\e^{-(k+1)/\alpha},\e^{-(k+1)/\alpha}]^{n-1}$ such that $Q_{k,1}' = [-\e^{-k/\alpha},\e^{-k/\alpha}]^{n-1}$ and
\[\gap(Q_{k,i}',Q_{k,j}') \geq (2\e)^{-k/\alpha} \qquad\text{for }i\neq j.\]

Set $E_{k,1}' = E_k'$. For each $i\in\{2,\dots,m_k\}$ (if any) let $\zeta_{k,i}$ be a similarity of $\R^n$ that maps $Q_{k,i}$ onto $[-1,1]^n$ and $\zeta_{k,i}'$ be a similarity of $\R^{n-1}$ that maps $Q_{k,i}'$ onto $[-1,1]^{n-1}$. Let also $E_{k,i}''$ and $g_{k,i} : E_{k,i}'' \to \R^n$ be the $(C',\alpha\beta)$-homogeneous subset of $[-1,1]^{n-1}$ and $\alpha$-H\"older map, respectively, of Case 1 for $\zeta_{k,i}(Q_{k,i}\cap E)$. Set $E_{k,i}' = (\zeta_{k,i}')^{-1}(E_{k,i}'')$ and
\[  E_{k+1}' = \bigcup_{i=1}^{m_k} E_{k,i}'. \]
Define $f_{k+1}: E_{k+1}' \to \R^n$ with $f_{k+1}|_{E_{k,1}'} = f_k|_{E_{k,1}'}$ and for $i=2,\dots,m_k$ (if any)
\[f_{k+1}|_{E_{k,i}'} = (\zeta_{k,i})^{-1}\circ g_{k,i}\circ\zeta_{k,i}'|_{E_{k,i}'}.\]
It is easy to see that the map $f_{k+1}: E_{k+1}' \to \R^n$ is $\alpha$-H\"older with $\Hold_{\alpha} f = L_1(n,C,\beta,s)$. Furthermore, the set $E_{k+1}'$ is $(C_1,\alpha\beta)$-homogeneous for some constant $C_1=C_1(n,C,\beta,s)$.

Thus, we construct a nested sequence of $(C_1,\alpha\beta)$-homogeneous sets $E_1'\subseteq E_2'\subseteq \cdots$ in $\R^{n-1}$ and a sequence of $\alpha$-H\"older maps $f_k : E_k' \to \R^n$ such that for each $k\in\N$,
\[ \Hold_{\alpha} f_k = L_1,\quad f_{k+1}|_{E_k'} = f_k|_{E_k'}\quad \text{and } f_{k}(E_k') = E_k.\]
Set  $E' = \bigcup_{k\in\N}E_k$. Since the H\"older constant is uniform, the sequence $f_k$ converges uniformly on compact sets to a map $f: E'\to \R^n$ which is $\alpha$-H\"older. Moreover, $E'$ is $(C_1,\alpha\beta)$-homogeneous and $f(E') = E$.
\end{proof}

Now we are ready to show Theorem \ref{thm:param}.

\begin{proof}[{Proof of Theorem \ref{thm:param}}]
Let us abbreviate $m/s=:\alpha$. Since H\"older maps extend to the closure of their domains with the same H\"older constant, and since $\dim_A(\overline{E}) = \dim_A(E)$ we may assume for the rest that $E$ is closed. By McShane's extension theorem (see e.g. \cite[VI 2.2, Theorem 3]{Stein}), it is enough to construct a set $E'\subseteq \R^m$ and a $\a$-H\"older map $f:E' \to E$ such that $E = f(E')$. Let $k=\lceil s \rceil$ be the smallest integer, bigger or equal to $s$. Fix $\beta\in(k-1,s)$ such that $\beta > \dim_A(E)$. Then there exists $C>1$ such that $E$ is $(C,\beta)$-homogeneous.

If $k<n$ then by Proposition \ref{prop:n-1} there exists a $\beta$-homogeneous set $A_1\subseteq \R^{n-1}$ and a Lipschitz surjective map $g_1:A_1\to E$. Proceeding inductively, for $i=2,\dots,n-k$ there exists $\beta$-homogeneous sets $A_i\subseteq\R^{n-i}$ and Lipschitz surjective maps maps $g_i:A_{i} \to A_{i-1}$. Thus, $g=g_1\circ\cdots\circ g_{n-m}$ is a Lipschitz map of $A_{n-k}\subseteq \R^k$ onto $E$ and it remains to produce a $\alpha\beta$-homogeneous set $E''\subseteq \R^m$ and a $\alpha$-H\"older surjective map $f:E'' \to A_{n-k}$. Hence, we may assume for the rest that $k=n$.

Suppose that $k=n$. Fix a number $\a_1\in(0,1)$ such that
\[ \frac{n-1}{s} < \a_1 <\frac{n-1}{\beta}.\]
Inductively, assuming we have defined numbers $\a_1,\dots,\a_{i}\in (0,1)$ for some $i\in\{1,\dots,n-m-1\}$, fix a number $a_{i+1}\in(0,1)$ such that
\[ \frac{n-i-1}{n-i} < \a_{i+1} < \min \left \{1, \frac{n-i-1}{\a_1\cdots \a_{i}\beta}\right \}.\]
An induction on $i$ shows that each $\a_i$ is well defined and that for each $i\in\{1,\dots,n-m\}$
\[\frac{n-i}{s} < \a_1\cdots \a_i < \frac{n-i}{\beta}.\]
Applying Proposition \ref{prop:n-1}, there exists a $\a_1\beta$-homogeneous set $E_1 \subseteq \R^{n-1}$ and a $\a_1$-H\"older map $f_1:\R^{n-1} \to \R^n$ with $f_1(E_1)= E$. Inductively, for $i =2,\dots,n-m$ there exists a $\a_1\cdots \a_{i}\beta$-homogeneous set $E_{i}\subseteq \R^{n-i}$ and a $\a_{i}$-H\"older map $f_{i}: E_{i} \to E_{i-1}$ such that $f_{i}(E_{i}) = E_{i-1}$. Let $\phi : \R^{m} \to \R^m$ with
\[ \phi(x) = |x|^{\a/(\a_1\cdots \a_{n-m}) -1}x.\]
Set $E'=\phi^{-1}(E_{n-m})$ and $f:E' \to E$ with
\[ f =  f_1\circ \cdots \circ f_{n-m} \circ \phi|_{E'}.\]
Since $\phi$ is $\a/(\a_1\cdots \a_{n-m})$-H\"older, it follows that $f$ is $\a$-H\"older and $f(E') = E$.
\end{proof}

\section{Proof of Theorem \ref{thm:bilip} (bi-Lipschitz surfaces)}\label{sec:BLthm}

To lay the groundwork for the proof of Theorem \ref{thm:bilip}, we first recall MacManus' cubical approximation of uniformly disconnected sets.
For every compact set $E\subseteq\RR^n$ and $\delta>0$, let $\mathcal{D}_{\d}(E)$ denote the collection of $n$-manifolds with boundary $M^n\subseteq\RR^n$ such that
$$\partial M \subseteq \mathscr{G}_{\delta}^{n-1}\quad\text{and}\quad \gap(\partial M,E)=\inf_{x\in\partial M}\inf_{y\in E}|x-y| \geq \d,$$ where $\mathscr{G}_\delta^{n-1}$ denotes the union of $(n-1)$-faces of cubes of side length $\delta$ defined in \S\ref{s:surfaces}.
MacManus stated and proved the following lemma in the case $n=2$ on \cite[p.~272]{MM}, and then commented on the general case in the first paragraph of \cite[p.~276]{MM}.

\begin{lemma}[{\cite[Lemma 2.3]{MM}}]\label{lem2}
For all $n\geq 2$ and $c>1$, there exists a constant $C = C(n,c)>1$ with the following property. If $E\subseteq \R^n$ is compact and $c$-uniformly disconnected, then for all $\delta>0$, there exists a finite collection $\mathcal{M}\subseteq \mathcal{D}_\delta(E)$ of pairwise disjoint manifolds intersecting $E$ such that
\begin{equation*} E \subseteq \bigcup_{M\in\mathcal{M}}{M},\end{equation*}
\begin{equation*}\delta \leq \diam{M} \leq C \delta\quad\text{for all }M\in\mathcal{M},\text{ and}\end{equation*}
\begin{equation*}\delta \leq \dist(x,E) \leq C \delta\quad\text{for all $M\in\mathcal{M}$ and $x\in \partial M$}.\end{equation*}
\end{lemma}

\begin{corollary}\label{cor:decomp}
For all $n\geq 2$ and $c>1$, there exist constants $C_0 = C_0(n,c)>1$ and $\e_0 = \e_0(n,c)>0$ with the following property. If  $E\subseteq \R^n$ is compact and $c$-uniformly disconnected, then for all $\varepsilon\in(0,\e_0),$ with $\e^{-1}\in \mathbb{N}$, there is an integer $N = N(n,c,\e)\geq 1$, a set $\mathcal{W}$ of finite words in $\{1,\dots,N\}$, and a family $\{M_w:w\in\mathcal{W}\}$ of $n$-manifolds with boundary such that:
\begin{enumerate}
\item The empty word is in $\mathcal{W}$, and for every word $w\in \W$,  there exists $N_w\in\{1,\dots,N\}$ such that $wi \in \W$ for all $i\in \{1,\dots, N_w\}$.
\item For all $w\in \W$, the associated manifold $M_w \in \mathcal{D}_{\e^{|w|}\diam{E}}(E)$ and $$\diam{M_w} \leq C_0\e^{|w|}\diam{E},$$ where $|w|$ denotes the length of $w$.
\item For all distinct $w,w' \in \W$ with $|w|=|w'|$,  $$\gap(M_w,M_{w'}) \geq \e^{|w|}\diam{E}.$$
\item For all $w\in\W$ and for all $i\in \{1,\dots,N_w\}$, we have $M_{wi} \subseteq M_w$ and
$$ \gap(M_{wi},\partial M_w) \geq C_0^{-1}\e^{|w|}\diam{E}.$$
\item For all $w\in \W$, the intersection $E \cap M_w \neq \emptyset$ and $\gap(\partial M_w, E) \geq \e^{|w|}\diam{E}$.
\item The set $E$ is the limit of the $k$-th level approximations: $$E =\bigcap_{k\geq 0} \bigcup_{\stackrel{w\in\mathcal{W}}{|w|=k}}M_w.$$
\end{enumerate}
\end{corollary}

\begin{proof} Let $n\geq 2$ and $c>1$. We will prove that the corollary holds with
\[  \e_0 = (2C)^{-1}, \qquad  C_0 = 10\sqrt{n}C\qquad\text{and}\qquad N = (\lceil C_0\rceil\e^{-1})^n,\]
where $C$ is the constant in Lemma \ref{lem2}.

Let $E\subseteq \R^n$ be compact and $c$-uniformly disconnected, and let $\varepsilon\in(0,\varepsilon_0)$. To ease notation, we may assume without loss of generality that $\diam E=1$. Choose an $n$-cube $M_{\emptyset}\subseteq \mathcal{D}_{1}(E)$ of side length $10$ such that $M_\emptyset$ contains $E$ and $$\gap(E,\partial M_{\emptyset}) \geq 1.$$ Then $M_\emptyset$ satisfies properties \emph{(2)} and \emph{(5)}.

Suppose that $M_w$ has been defined for some word $w$ so that $M_w$ satisfies both properties \emph{(2)} and \emph{(5)}. Applying Lemma \ref{lem2} to $E\cap M_w$ with $\delta=\e^{|w|+1}$, we can find a finite collection $\{M_{w1},\dots,M_{wN_w}\}\subseteq\mathcal{D}_{\varepsilon^{|w|+1}}(E\cap M_w)$ such that
\begin{equation}\label{eq:1} E\cap M_w \subseteq \bigcup_{i=1}^{N_w}{M_{wi}},\end{equation}
\begin{equation}\label{eq:2}\varepsilon^{|w|+1} \leq \diam{M_{wi}} \leq C \varepsilon^{|w|+1}\quad\text{for all }1\leq i\leq N_w,\text{ and}\end{equation}
\begin{equation}\label{eq:3}\varepsilon^{|w|+1} \leq \dist(x,E\cap M_w) \leq C \varepsilon^{|w|+1}\quad\text{for all $1\leq i\leq N_w$ and $x\in \partial M_{wi}$}.\end{equation}

By property \emph{(2)}, $M_w$ consists of at most $\lceil C_0\rceil^n$ cubes in $\mathscr{G}_{\e^{|w|}\diam{E}}$. Since $\e^{-1}$ is an integer, each cube in $\mathscr{G}_{\e^{|w|}\diam{E}}$ consists of exactly $\e^{-n}$ cubes in $\mathscr{G}_{\e^{|w|+1}\diam{E}}$. Therefore, there are at most $(\lceil C_0\rceil\e^{-1})^n$ cubes $Q \in \mathscr{G}_{\e^{|w|+1}\diam{E}}$ contained in $M_w$. Since each $M_{wi}$ is a union of cubes $Q \in \mathscr{G}_{\e^{|w|+1}\diam{E}}$ contained in $M_w$, and $M_{w1},\dots,M_{wN_w}$ are mutually disjoint, we have $N_w \leq N$.

Furthermore, for each $M_{wi}$, property \emph{(2)} follows from (\ref{eq:2}) and property \emph{(3)} follows from the fact that the sets $M_{wi}$ are disjoint and belong in $\mathcal{D}_{\e^{|w|+1}}(E\cap M_w)$. Since $\e<(2C)^{-1}$, by (\ref{eq:3})
\begin{align*}
\gap(M_{wi},\partial M_w) &\geq \gap(\partial M_w, E\cap M_w) - \sup_{x\in \partial M_{wi}}\dist(x,E\cap M_w)\\
&\geq\gap(\partial M_w, E) - \sup_{x\in \partial M_{wi}}\dist(x,E\cap M_w)\\
&\geq \e^{|w|} - C\e^{|w|+1} > \frac12\e^{|w|}
\end{align*}
and property \emph{(4)} holds. For property \emph{(5)},
\begin{align*}
\gap(\partial M_{wi}, E) &= \min\{\gap(\partial M_{wi}, E\cap M_w),\gap(\partial M_{wi}, E\setminus M_w) \}\\
&\geq \min\{\gap(\partial M_{wi}, E\cap M_w),\gap(M_{wi}, \partial M_w) \} \geq \e^{|w|+1}.
\end{align*}
Finally, each $M_w$ intersects $E$ and by (\ref{eq:1}), $E \subseteq \bigcup_{\stackrel{w\in\mathcal{W}}{|w|=k}}M_w$. Therefore, property \emph{(6)} holds and the proof is complete.
\end{proof}

We first prove Theorem \ref{thm:bilip} in the special case that $m=n-1$ and $E$ is compact.

\begin{proposition}[bi-Lipschitz parameterization for compact sets, in codimension one]\label{sec:tree} For all $n\geq 2$, $c>1$, $C>1$, and $0\leq s<n-1$, there exists a constant $L=L(n,c,C,s)\geq \sqrt{2}$ with the following property. If $E\subseteq\RR^n$ is compact, $c$-uniformly disconnected, and $(C,s)$-homogeneous, then there is an $L$-bi-Lipschitz embedding $f:\R^{n-1} \to \R^n$ such that $E\subseteq f([-1,1]^{n-1})$ and $f|_{\RR^{n-1}\setminus [-1,1]^n}$ is an isometric embedding.\end{proposition}

\begin{proof} Let $n\geq 2$, $c>1$, $C>1$, $s\in[0,n-1)$, and assume that $E\subseteq\RR^n$ is compact, $c$-uniformly disconnected, and $(C,s)$-homogeneous. Applying similarities to the domain and range of the embedding, we may assume without loss of generality that $E$ contains the origin and has diameter 1. For the rest of the proof, fix an integer $k_0$ such that
\[ 2^{k_0} \geq 8\sqrt{n}.\]
The proof now breaks up into three steps. In \textsc{Step 1}, we construct a surface containing the set $E$. In \textsc{Step 2}, we build a homeomorphism between the surface and $\RR^{n-1}$. Then, in \textsc{Step 3}, we verify that the homeomorphism is bi-Lipschitz.

\bigskip\noindent\textsc{Step 1}. We will use Corollary \ref{cor:decomp} to build a tree-like surface\footnote{Constructions of tree-like surfaces are by now classical. For instance, see \cite[Figure 2.4.16]{Rushing}.} $\mathcal{S}$ that contains $E$. Set
\[\e^{-1} := 1+\max\{\lceil (3^{n-1}(2C_0)^{\beta}C)^{\frac1{n-1-\beta}} \rceil, \lceil \e_0^{-1}\rceil\},\]
where $\e_0$ and $C_0$ are the constants of Corollary \ref{cor:decomp}. By the $(C,\beta)$-homogeneity of $E$ and choice of $\varepsilon$,
\begin{equation}\label{eq:epsilon}
C\left ( \frac{2C_0\e^k}{\e^{k+1}} \right )^{\beta} = C(2C_0)^{\beta} \e^{-\beta} \leq (3\e)^{1-n}.
\end{equation}
Let $\{M_w:w\in\mathcal{W}\}$ be the family of manifolds with boundary associated to $\varepsilon$ given by Corollary \ref{cor:decomp}. By our assumption on the size and position of $E$, we may assume without loss of generality that the initial manifold $M_{\emptyset}=[-5,5]^n$ (see the proof of Corollary \ref{cor:decomp}). Define $$A_\emptyset:=[-5,5]^{n-1}\times\{5\}\subseteq\partial M_{\emptyset},$$ and for each word $w\in\mathcal{W}$ with $|w|\geq 1$, choose an $(n-1)$-cube $$A_w\subseteq\mathscr{G}^{n-1}_{\varepsilon^{|w|}}\cap \partial M_w$$ of side length $\varepsilon^{|w|}$ and boundary in $\mathscr{G}^{n-2}_{\varepsilon^{|w|}}$ . Let $x_w$ denote the center of $A_w$. By properties \emph{(2)} and \emph{(6)} of $\{M_w:w\in\mathcal{W}\}$, the set $$E= \lim_{l\rightarrow\infty} \bigcup_{\stackrel{w\in\mathcal{W}}{|w|=l}} A_w$$ in the Hausdorff topology. Thus, we aim to constructing a sequence of intermediate surfaces that pass through successive generations of the $(n-1)$-cubes $A_w$.

\emph{Base Case.} Define $$ x_{\emptyset} := (0,\dots,0,5)\quad \text{and}\quad y_\emptyset:=(0,\dots,0,10).$$ Note that $x_\emptyset$ is the center of $A_\emptyset$. Let $\g_{\emptyset}$ denote the line segment from $y_\emptyset$ to $x_{\emptyset}$ and let $\tau_{\emptyset}$ denote the tube around $\gamma_\emptyset$ of width $\delta_\emptyset=2^{-k_0}$
\[ \tau_{\emptyset} = \tube_{\d_{\emptyset}}(\g_{\emptyset}) = [-2^{-k_0-1},2^{-k_0-1}]^{n-1} \times [5,10]. \] With the specified orientation on $\gamma_\emptyset$, $$\entrance(\tau_\emptyset)=[-2^{-k_0-1},2^{-k_0-1}]^{n-1}\times \{10\}, \quad\exit(\tau_\emptyset)= [-2^{-k_0-1},2^{-k_0-1}]^n\times\{5\}=A_\emptyset.$$ Next, define the punched-out plane,
\begin{equation}\label{eq:P} \mathcal{P} := (\R^{n-1}\setminus [-1,1]^{n-1})\times\{10\}\end{equation}
and an auxiliary $(n-1)$-cube, $S_{\emptyset} := [-1,1]^{n-1} \times \{10\}$.
The union
\[ \mathcal{S}^{(0)}:=\mathcal{P} \cup (S_{\emptyset}\setminus \entrance(\tau_{\emptyset})) \cup \partial \tau_\emptyset\]
denotes the 0th level approximation of surface $\mathcal{S}$.

\emph{Inductive Step.} Let $w\in\W$ and assume that we have defined $\tau_{w} = \tube_{\d_w}(\g_w)$ for some arc $\gamma_w$ and $\d_w < \e^{|w|}/3$ such that $$\tau_{w}\cap \overline{M_w}=\exit(\tau_w)\subseteq A_w\subseteq\partial M_w.$$
First, divide $A_w$ into $3^{n-1}$ congruent $(n-1)$-cubes of side length $\frac13 \e^{|w|}$, and let $\hat{A}_w$ denote the central subcube in the division. Define $\mathcal{K}_w$ to be the unique $n$-cube contained in $M_w$, which has $\hat{A}_w$ as an $(n-1)$-face. Let $\tilde{A}_w$ denote the $(n-1)$-face in $\mathcal{K}_w$ that is antipodal to $\hat{A}_w$. Second, divide $\tilde{A}_w$ into $(3N)^{n-1}$-many $(n-1)$-cubes of side length $(1/9N) \e^{|w|}$, and choose subcubes $S_{w1},\dots,S_{wN_w}$ in the division that are mutually disjoint and satisfy $S_{wi}\cap (\partial \mathcal{K}_w \setminus \tilde{A}_w) = \emptyset$. Thus, the $(n-1)$-cubes $S_{wi}$ lie on the relative interior of $\tilde{A}_w$, \begin{equation*}\gap(S_{wi},S_{w_j}) \geq \frac{1}{9N}\e^{|w|}\quad\text{when $i\neq j$, and}\end{equation*}
\begin{equation*}\gap(S_{wi},\partial \tilde{A}_w) \geq \frac{1}{9N}\e^{|w|}\quad\text{for all $i\in\{1,\dots,N_w\}$}.\end{equation*} For each $i=1,\dots,N_w$, let $y_{wi}$ denote the center of $S_{wi}$.

\begin{figure}
\includegraphics[scale=0.7]{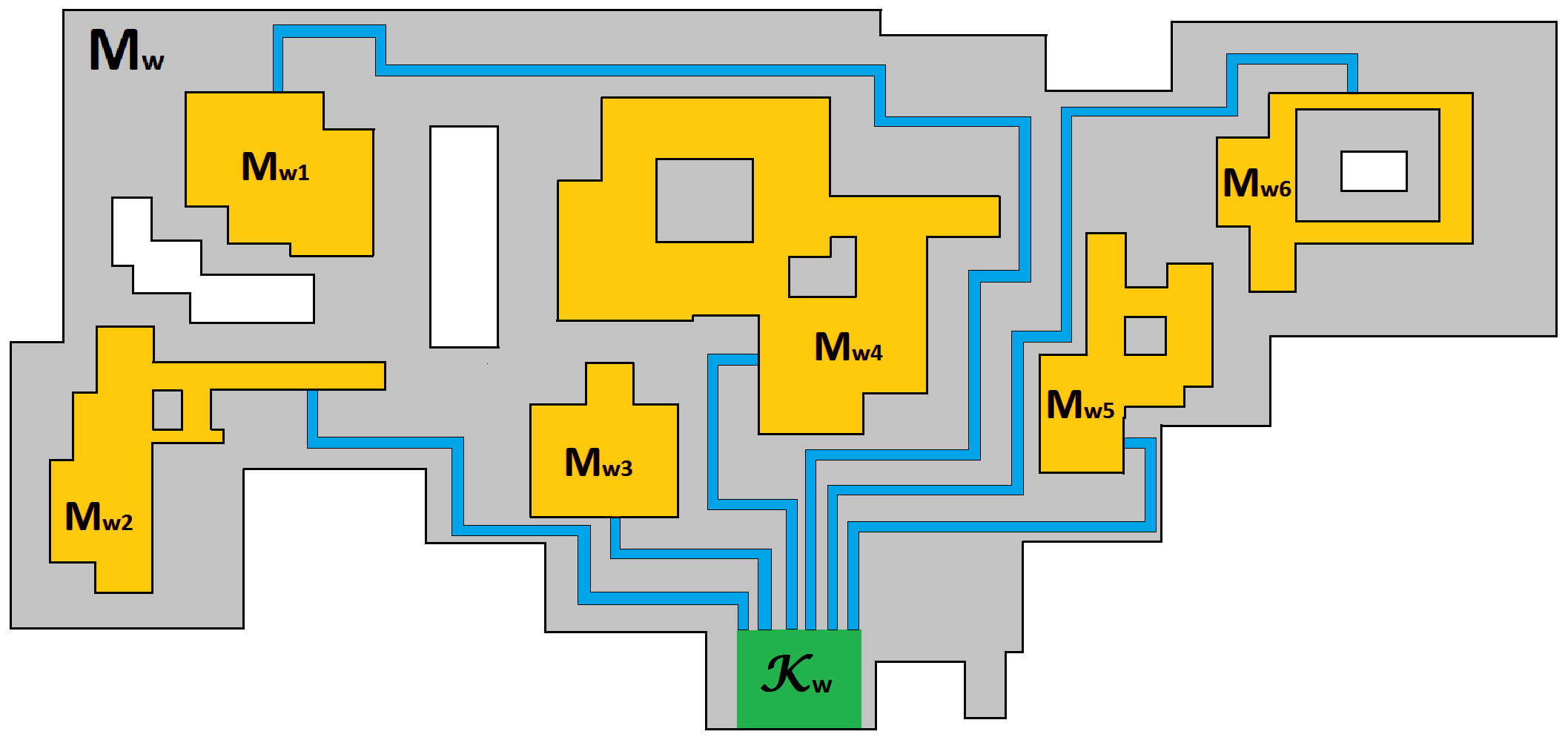}
\caption{The subtree of $\mathcal{S}$ inside $M_w$.}\label{fig:tree}
\end{figure}

We now select a sequence of arcs connecting the points $y_{wi}$ to the centers $x_{wi}$ of the next generation $(n-1)$-cubes $A_{wi}$; see Figure \ref{fig:tree}. First, choose a polygonal arc $\g_{w1}\subseteq\mathscr{G}^1_{\e^{|w|+1}/2}$ with endpoints in $\mathscr{G}^0_{\e^{|w|+1}/2}$ that lies (except at its endpoints) in the interior of the manifold
\[ M_w\setminus \left ( \mathcal{K}_w\cup \bigcup_{j=1}^{N_w} M_{wj}\right )\]
and joins $y_{w1}$ to $x_{w1}$. Proceeding inductively, for each $i=2,\dots,N_w$, choose a polygonal arc $\g_{wi}\subseteq\mathscr{G}^1_{2^{-i-1}\e^{|w|+1}}$ with endpoints in $\mathscr{G}^0_{2^{-i-1}\e^{|w|+1}}$ that lies (except at its endpoints) in the interior of the manifold
\[M_w \setminus \left (\mathcal{K}_w\cup\bigcup_{j=1}^{N_w} M_{wj} \cup \bigcup_{j=1}^{i-1} \g_{wj}\right )\]
and joins $y_{wi}$ to $x_{wi}$. Since each $\g_{wi}\subseteq \mathscr{G}^1_{2^{-N-1}\e^{|w|+1}}$, $\g_{wi}\subseteq M_{w}$ and $\diam{M_w} \leq C_0\e^{|w|}$, by Remark \ref{rem:length}, the length of each $\g_{wi}$ is uniformly bounded:
\begin{equation}\label{eq:length}
{H}^1(\g_{wi}) \leq C_1 \e^{|w|}
\end{equation}
for some $C_1>0$ depending only on $\e$, $N$ and $C_0$ (thus, only on $n$, $c$, $C$ and $s$).

For each $i=1,\dots,N_w$, let $\d_{wi} = 2^{-N-k_0}\e^{|w|+1}$ and $\tau_{wi} = \tube_{\d_{wi}}(\g_{wi})$, oriented so that $\entrance(\tau_{wi})$ lies on $S_{wi}$, while  $\exit(\tau_{wi})$ lies on $\partial M_{wi}$. Define $\Delta_w$ to be a punched-out $(n-1)$-cube, with $N_w$-many $(n-1)$-cubical holes, say
\begin{equation}\label{eq:Delta}\D_w:=\overline{\tilde{A}_w \setminus\bigcup_{i=1}^{N_w}S_{wi}}.\end{equation} Also, define the topological $(n-1)$-annulus,
\begin{equation}\label{eq:A} \mathcal{A}_w := \overline{ ( S_w \setminus \entrance(\tau_w) ) \cup \side(\tau_w) \cup  ( \partial\mathcal{K}_w \setminus (\tilde{A}_w\cup \exit(\tau_w))  )} .\end{equation}
The $k$-th level approximation of $\mathcal{S}$ is
\[ \mathcal{S}^{(k)} := \mathcal{P}\cup \bigcup_{{\stackrel{w\in\mathcal{W}}{|w|\leq k}}}(\D_w \cup \mathcal{A}_w).\]

\emph{End Step.} Define $\mathcal{S}$ to be the closure of $\bigcup_{k=0}^{\infty}\mathcal{S}^{(k)}$; that is,
\[ \mathcal{S} := E \cup \mathcal{P}\cup \bigcup_{w\in\W}(\D_w \cup \mathcal{A}_w).\] See (\ref{eq:P}), (\ref{eq:Delta}), and (\ref{eq:A}).

\bigskip\noindent\textsc{Step 2}. We will construct a homeomorphism $f : \R^{n-1} \to \mathcal{S}$ that maps $\RR^n\setminus[-1,1]^{n-1}$ onto the punched-out plane $\mathcal{P}$ , isometrically, and maps $[-1,1]^{n-1}$ onto $\mathcal{S} \setminus \mathcal{P}$. Afterwards, in \textsc{Step 3}, we verify that $f$ is actually bi-Lipschitz.

Let us decompose $[-1,1]^n$ into subsets corresponding to the preimages of parts of $\mathcal{S}\setminus \mathcal{P}$. First, assign $M_{\emptyset}' := [-1,1]^{n-1}$. We proceed by induction. Suppose that for some $w\in\W$, we have defined sets $M_{w}' = \mathcal{C}^{n-1}(z_w;\e^{|w|})$. By $(C,\beta)$-homogeneity of $E$ and (\ref{eq:epsilon}), we have $N_w\leq (3\e)^{1-n}$. Thus, we can locate cubes $M_{wi}' := \mathcal{C}^{n-1}(z_{wi};\e^{-|w|-1})$ for all $i=1,\dots,N_w$ so that for all distinct $i,j\in\{1,\dots,N_w\}$,
\begin{itemize}
\item $M_{wi}' \subseteq M_{w}'$,
\item $\gap (M_{wi}',M_{wj}') \geq \e^{|w|+1}$, and
\item $\gap (M_{wi}',\partial M_{w}') \geq \e^{|w|+1}$.
\end{itemize}
For each $w\in\mathcal{W}$, let
\[\D_w' = \overline{(M_w'\setminus\mathcal{A}_w') \setminus \bigcup_{i=1}^{N_w}M_{wi}'}\qquad\text{and}\qquad\mathcal{A}_{w}' = \mathcal{A}^{n-1}(z_{w}; \e^{|w|}-\frac12 \e^{|w|+1},\e^{|w|}).\]
Finally, define $E' := \bigcap_{n=0}^{\infty}\bigcup_{w\in\W , |w|=n} M_w'$.

By (\ref{eq:length}) and Lemma \ref{lem:blthick} (replacing $n$ with $n-1$), there exists $L_1 = L_1(n,c,C,\beta)>1$ such that for every $w\in\W$, there exists a homeomorphism $\phi_w : \mathcal{A}^{n-1}(0;1-\frac{\e}{2},1) \to  \mathcal{A}_w$ such that
\begin{itemize}
\item $\e^{-|w|}\phi_w$ is $L_1$-bi-Lipschitz and orientation preserving;
\item $\phi_w|_{\partial \C^{n-1}(0;1)}$ is a similarity that maps $\partial \C^{n-1}(0;1)$ onto the relative boundary of $S_w$;
\item $\phi_w|_{\partial \C^{n-1}(0;1-\frac{\e}{2})}$ is a similarity that maps $\partial \C^{n-1}(0;1-\frac{\e}{2})$ onto $\mathcal{A}_{w}\cap \D_w$.
\end{itemize}
Now define $f|_{\mathcal{A}_w'} : \mathcal{A}_w' \to \mathcal{A}_w$ by setting $$f|_{\mathcal{A}_w'}(x) := \phi_w(\e^{-|w|}(x-z_w)).$$ Then $f|_{\mathcal{A}_w'}$ is an $L_1$-bi-Lipschitz orientation preserving homeomorphism of $\mathcal{A}_w'$ onto $\mathcal{A}_w$. Applying a standard extension argument (e.g., see \cite[Proposition 3.6]{Vext}), one can show that there exists $L_2=L_2(n,c,C,\beta)>1$ so that $f$ extends to a $L_2$-bi-Lipschitz map on each $\D_w'$, with $f(\D_w') = \D_w$. Finally, since $E$ is contained in the closure of $\mathcal{S}$, the map $f$ extends uniquely on $E'$ as a homeomorphism and maps $E'$ onto $E$.

Therefore, we have obtained a homeomorphism $f:\R^{n-1} \to \mathcal{S}$ such that $f([-1,1]^{n-1})= \mathcal{S}\setminus \mathcal{P}$ and $f(E')=E$. For each $w\in\W$, define
\[ \mathcal{S}_w := \mathcal{A}_w\cup\D_w \quad\text{ and }\quad \mathcal{S}_w' := \mathcal{A}_w'\cup\D_w'.\]
By construction, $f(\mathcal{S}_w') = \mathcal{S}_w$.

\bigskip\noindent\textsc{Step 3}. It remains to show that $f$ is $L$-bi-Lipschitz for some $L = L(n,c,C,\beta)>1$. By the previous steps, there exists $L_0 = L_0(n,c,C,\beta)>1$ such that for all $w\in\W$ and distinct $i,j\in\{1,\dots,N_w\}$:
\begin{enumerate}
\item The restrictions $f|_{\mathcal{S}_w'\cup \mathcal{S}_{wi}'}$ and $f|_{\R^{n-1}\setminus \bigcup_{i=1}^{N_{\emptyset}}M_{i}'}$ are $L_0$-bi-Lipschitz.
\item $L_0^{-1}\e^{|w|} \leq \diam{\mathcal{S}_w'} \leq L_0 \e^{|w|}$ and $L_0^{-1}\e^{|w|} \leq \diam{\mathcal{S}_w} \leq L_0 \e^{|w|}$.
\item $L_0^{-1}\e^{|w|} \leq \diam{M_{w}'} \leq L_0 \e^{|w|}$ and $L_0^{-1}\e^{|w|} \leq \diam{M_{w}} \leq L_0 \e^{|w|}$.
\item $L_0^{-1}\e^{|w|} \leq \gap(M_{wi}',M_{wj}') \leq L_0 \e^{|w|}$ and $L_0^{-1}\e^{|w|} \leq \gap(M_{wi},M_{wj}) \leq L_0 \e^{|w|}$.
\end{enumerate}
Below, we say that two points $x,y\in\R^{n-1}$ are \emph{separated by $\mathcal{S}_w'$} for some $w\in\mathcal{W}$ if neither $x$ nor $y$ is contained in $\mathcal{S}_w'$ and any curve in $\R^{n-1}$ joining $x$ and $y$ intersects $\mathcal{S}_w'$. Also, given $a,b>0$, we write $a\lesssim b$ to denote that $a \leq C^* b$ for some $C^* = C^*(n,c,C,\beta)>1$ and $a\sim b$ to denote that $a\lesssim b$ and $b\lesssim a$.

To show that $f$ is bi-Lipschitz, fix $x,y \in \R^{n-1}$. First suppose that $x\in \R^{n-1} \setminus [-1,1]^{n-1}$. On one hand, if $y\in \mathcal{S}_{\emptyset}'$, then $$|x-y| \sim |f(x)-f(y)|$$ by (1). On the other hand, if $y \in \bigcup_{i=1}^{N_{\emptyset}}M_{i}'$, then
\[|x-y| \sim 1 + \dist(x,M_{\emptyset}') \sim 1 + \dist(f(x),M_{\emptyset}) \sim |f(x)-f(y)|\] by (2).
In both cases, $|f(x)-f(y)| \sim |x-y|$. Therefore, to complete the proof, we may assume that $x,y\in [-1,1]^{n-1}$. There are two alternatives.

\emph{Case 1.} Suppose that $x$ and $y$ are not separated $\mathcal{S}_w'$ for any $w\in\mathcal{W}$. Then there exists $w\in \mathcal{W}$ and $i\in \{1,\dots,N_w\}$ such that $x,y \in \mathcal{S}_w'\cup \mathcal{S}_{wi}'$. Hence $$|f(x)-f(y)| \sim |x-y|$$ by (1).

\emph{Case 2.} Suppose that $x$ and $y$ are separated by $\mathcal{S}_w$ for some $w\in\mathcal{W}$. Let $w_0$ be the minimal word with the property that $\mathcal{S}_{w_0}$ separates $x$ and $y$. That is, if $\mathcal{S}_w$ separates $x$ and $y$, then $w = w_0 u$. Since $x,y\in [-1,1]^{n-1}$, we have $w_0 \neq \emptyset$. Hence
\[|x-y|\sim \e^{|w_0|}\sim |f(x)-f(y)|\] by (2), (3), and (4).
This completes the proof that $f$ is bi-Lipschitz.
\end{proof}

We now derive Theorem \ref{thm:bilip} from Proposition \ref{sec:tree}.

\begin{proof}[{Proof of Theorem \ref{thm:bilip}}]
Let $E\subseteq \R^n$ be $c$-uniformly disconnected and $(C,\beta)$-homogeneous for some $n,m\in\N$, $C>1$, $c>1$ and $\beta<m\leq n-1$. Since bi-Lipschitz maps extend to the closure of their domain and $\overline{E}$ is also $(C,\beta)$-homogeneous, $E$ can be assumed closed.

Suppose first that $m=n-1$. By Proposition \ref{sec:tree} we may assume that $E$ is unbounded.  Fix distinct points $x_1,x_2\in E$. For each $k\in\N$, let $E_k$ be the set $E_{x_1,2^k|x_1-x_2|}$ appearing in the definition of uniform disconnectedness.  Note that each $E_k$ is compact, $c$-uniformly disconnected and $(C,\beta)$-homogeneous. For each $k\in\N$, by Proposition \ref{sec:tree}, there exists an $L$-bi-Lipschitz embedding $f_k:\R^{n-1} \to \R^n$ with $L=L(n,c,C,\beta)>1$ and $E_k \subseteq f_k(\R^{n-1})$. Applying appropriate similarities, we may assume that $f_k(0,\dots,0) = x_1$ for all $k\in\N$. By the Arzel\`a-Ascoli Theorem, there exists a subsequence $f_{k_j}$ that converges uniformly on compact sets to an $L$-bi-Lipschitz embedding $f:\R^{n-1} \to \R^n$. For each $x\in E$, the sequence $f_{k_j}^{-1}(x)$ converges to a point $f^{-1}(x)$ in $\R^{n-1}$, and consequently, $E\subseteq f(\R^{n-1})$.

Suppose now that $m<n-1$. Set $c_0=c$ and $C_0 = C$. By the codimension 1 case, for each $k=1,\dots,n-m$, there exist $L_k=L_k(n,c_{k-1},C_{k-1},s)>1$, $c_{k}=c_{k}(n,c_{k-1},C_{k-1},s)>1$, $C_{k} = C_{k}(n,c_{k-1},C_{k-1},s)>1$, an $L_k$-bi-Lipschitz embedding $f_k: \R^{n-k} \to \R^{n-k+1}$ and a $c_k$-uniformly disconnected and $(C_{k},\beta)$-homogeneous set $E_k\subseteq\R^{n-k}$ such that $f(E_k) = E_{k-1}$. The map $f:\R^{m}\to\R^n$ with $f = f_1\circ\cdots\circ f_{n-m}$ is $(L_{n-m}\cdots L_1)$-bi-Lipschitz and maps $E_{n-m}$ onto $E$.
\end{proof}

\begin{remark}\label{rem:QS} Assume that $E\subseteq\RR^n$ is compact and $c$-uniformly disconnected. It is possible to modify the proof of Proposition \ref{sec:tree} to produce a quasisymmetric map $f:\RR^{n-1}\rightarrow\RR^n$ whose image contains $E$. To carry this out, first repeat the construction of the surface in Step 1 with the alternative parameter \[\e^{-1}:= 1+\max\{\lceil 2C_0 +1 \rceil, \lceil\e_0^{-1}\rceil\}. \] Then use arguments similar to \cite{MM} or \cite[Theorem 6.3]{Vais1} to parameterize the surface containing $E$ by a quasisymmetric map. To extend the proof to unbounded sets, apply an the Arzel\`a-Ascoli Theorem for quasisymmetric maps \cite[Corollary 10.30]{Heinonen}. Theorem \ref{thm:QS} may be derived from the codimension 1 case in the same way that Theorem \ref{thm:bilip} is derived from Proposition \ref{sec:tree}.
\end{remark}

\part{Geometry of measures}

In this part of the paper, we prove Theorems \ref{thm:A}, \ref{thm:B}, and \ref{thm:C}, which identify conditions on the lower and upper Hausdorff densities that guarantee a Radon measure is either carried by or singular to H\"older curves or surfaces. For the statements of these theorems, see \S1.1 in the introduction. The main tools that we use are three parameterization theorems from Part I: Theorem \ref{t:diam^s}, Theorem \ref{thm:param}, and Theorem \ref{thm:bilip} (in the form of Corollary \ref{cor:bilip}). The proof of Theorem \ref{thm:A} is given in \S\ref{s:proof-A} and the proof of Theorems \ref{thm:B} and \ref{thm:C} is given in \S\ref{s:proof-BC}.

\section{Points of extreme lower density (Proof of Theorem \ref{thm:A})}
\label{s:proof-A}

The proof of the first part of Theorem A---Radon measures are singular to H\"older curves on sets of vanishing lower density---uses the relationship between lower Hausdorff densities and packing measures. The argument that we present below closely follows \cite[\S2]{BS1}, which focused on Lipschitz images. To fix conventions, we recall the definition of  $s$-dimensional Hausdorff measure $\Haus^s$ and $s$-dimensional packing measure $\Pack^s$, each of which are Borel regular metric outer measures on $\RR^n$. In the top dimension ($s=n$), the measures $\Haus^n$ and $\Pack^n$ coincide and are a constant multiple of Lebesgue measure on $\RR^n$. For a proof of these facts and further background, see \cite{Falconer} or  \cite{Mattila}.

\begin{definition}[Hausdorff and packing measures in $\RR^n$] Let $s\geq 0$ be a real number. Let $E,E_1,E_2,\dots$ denote sets in $\RR^n$. The \emph{$s$-dimensional Hausdorff measure} $\Haus^s$ is defined by $\Haus^s(E)=\lim_{\delta\rightarrow 0} \Haus^s_\delta(E)$, where \begin{equation*} \Haus^s_\delta(E)=\inf\left\{\sum_i (\diam E_i)^s:E\subseteq\bigcup_i E_i,\ \diam E_i\leq \delta\right\}.\end{equation*} The \emph{$s$-dimensional packing premeasure} $P^s$ is defined by $P^s(E)=\lim_{\delta\rightarrow 0} P^s_\delta(E)$, where \begin{equation*}P^s_\delta(E)=\sup\left\{\sum_{i}(2r_i)^s: x_i\in E, 2r_i\leq \delta, i\neq j\Rightarrow B(x_i,r_i)\cap B(x_j,r_j)=\emptyset\right\}.\end{equation*} The \emph{$s$-dimensional packing measure} $\Pack^s$ is defined by \begin{equation*}\Pack^s(E)=\inf\left\{\sum_i P^s(E_i):E=\bigcup_i E_i\right\}.\end{equation*}\end{definition}

\begin{lemma}[see {\cite[Proposition 2.2]{Falconer}}] \label{l:bounds} Let $A\subseteq\RR^n$ be a Borel set, let $\nu$ be a finite Borel measure on $\RR^n$, and let $0<\lambda<\infty$. \begin{itemize} \item If $\limsup_{r\downarrow 0} \dfrac{\nu(B(x,r))}{r^s} \leq \lambda$ for all $x\in E$, then $\Haus^s(E) \geq \nu(E)/\lambda$.
\item If $\limsup_{r\downarrow 0} \dfrac{\nu(B(x,r))}{r^s} \geq \lambda$ for all $x\in E$, then $\Haus^s(E) \leq 2^s\nu(E)/\lambda$.
\item If $\liminf_{r\downarrow 0} \dfrac{\nu(B(x,r))}{r^s} \leq \lambda$ for all $x\in E$, then $\Pack^s(E) \geq 2^s\nu(E)/\lambda$.
\item If $\liminf_{r\downarrow 0} \dfrac{\nu(B(x,r))}{r^s} \geq \lambda$ for all $x\in E$, then $\Pack^s(E) \leq 2^s\nu(E)/\lambda$.
\end{itemize}\end{lemma}

It is well known that H\"older continuous maps do not increase Hausdorff measures too severely. The same phenomenon is also true for packing measures. We include a proof of the following lemma for the reader's convenience.

\begin{lemma}\label{l:lpack} Let $E\subseteq\RR^m$. If $f:E\rightarrow\RR^n$ is (1/s)-H\"older, then \begin{equation*}P^{st}(f(E))\leq 2^{(s-1)t}(\Hold_{1/s} f)^{st}\, P^{t}(E)\quad\text{and}\quad \Pack^{st}(f(E))\leq 2^{(s-1)t}(\Hold_{1/s} f)^{st}\,\Pack^t(E),\end{equation*} where $\Hold_{1/s} f$ denotes the $(1/s)$-H\"older constant of $f$. \end{lemma}

\begin{proof}Assume that $P^t(E)<\infty$ and $f:E\rightarrow\RR^n$ satisfies $|f(x)-f(y)|\leq H |x-y|^{1/s}$ for all $x,y\in E$. Given $\varepsilon>0$, pick $\eta>0$ such that $P^t_\eta(E)\leq P^t(E)+\varepsilon$. Fix $\delta>0$ such that $$ 2^{1-s} \left(\frac{\delta}{H}\right)^s\leq \eta$$  and let $\{B^n(f(x_i),r_i):i\geq 1\}$ be an arbitrary disjoint collection of balls in $\RR^n$ centered in $f(E)$ such that $2r_i\leq \delta$ for all $i\geq 1$. By the H\"older condition on $f$, \begin{equation*}f(B^m(x_i,(r_i/H)^s))\subseteq B^n(f(x_i),r_i)\quad\text{for all }i\geq 1.\end{equation*} Thus $\{B^m(x_i,(r_i/H)^s):i\geq 1\}$ is a disjoint collection of balls with centers in $E$ with $$2\left(\frac{r_i}{H}\right)^s \leq 2^{1-s} \left(\frac{\delta}{H}\right)^s\leq \eta.$$ Hence \begin{equation*} \sum_{i=1}^\infty (2r_i)^{st} = 2^{(s-1)t}H^{st}\sum_{i=1}^\infty \left(2\left(\frac{r_i}{H}\right)^s\right)^t \leq 2^{(s-1)t}H^{st}P^t_\eta(E)\leq 2^{(s-1)t}H^{st}(P^t(E)+\varepsilon).\end{equation*} Taking the supremum over all $\delta$-packings of $f(E)$, we obtain $$P^{st}_\delta(f(E))\leq 2^{(s-1)t}H^{st}(P^t(E)+\varepsilon).$$ Therefore, letting $\delta\rightarrow 0$ and $\varepsilon\rightarrow 0$, $P^{st}(f(E))\leq 2^{(s-1)t}H^{st}P^t(E)$. The corresponding inequality for the packing measure $\Pack^s$ follows immediately from the inequality for $P^s$.\end{proof}

The following lemma contains the first half of Theorem A. The special case $s=m$  appeared previously in \cite[Lemma 2.7]{BS1}. When the measure is of the form $\mu=\Haus^s\res E$ for some $s$-set $E\subseteq\RR^n$, this result also follows from \cite[Theorem 3.2]{MM1993}.

\begin{lemma}\label{l:pu} Let $1\leq m\leq n-1$ be integers and let $s\in[m,n]$. If $\mu$ is a Radon measure on $\RR^n$, then $$\underline{\mu}^s_{\,0}:=\mu\res\left\{x\in\RR^n:\liminf_{r\downarrow 0}\frac{\mu(B(x,r))}{r^s}=0\right\}$$ is singular to $(m/s)$-H\"older $m$-cubes.\end{lemma}

\begin{proof} For a large radius $R>0$, let $A_R=\{x\in B(0,R):\liminf_{r\downarrow 0} r^{-s}\mu(B(x,r))=0\}$ and $\nu_R=\mu\res A_R$. Then $\nu_R$ is a finite Borel measure.
Let $f:[0,1]^m\rightarrow\RR^n$ be an arbitrary $(m/s)$-H\"older continuous map. By Lemma \ref{l:lpack}, $$\Pack^s(f([0,1]^m)) \leq 2^{((s/m)-1)m}(\Hold_{m/s} f)^s\,\Pack^m([0,1]^m)<\infty.$$ Let $\lambda>0$. Because $\liminf_{r\downarrow 0} r^{-s}\mu(B(x,r))=0\leq \lambda$ for all $x\in A_R\cap f([0,1]^m)$, we have \begin{equation*}2^s\mu(A_R\cap f([0,1]^m)) \leq \lambda \Pack^s(A_R\cap f([0,1]^m))\leq \lambda \Pack^s(f([0,1]^m))\end{equation*} by Lemma \ref{l:bounds}. Then, letting $\lambda\rightarrow 0$, we obtain $\nu_R(f([0,1]^m))=\mu(A_R\cap f([0,1]^m))=0$. Therefore, since measures are continuous from below, $$\underline{\mu}^s_{\,0}(f([0,1]^m))=\lim_{R\uparrow \infty}\nu_R(f([0,1]^m))=0$$ for every $(m/s)$-H\"older continuous map $f:[0,1]^m\rightarrow \RR^n$. In other words, the measure $\underline{\mu}^s_{\,0}$ is singular to $(m/s)$-H\"older $m$-cubes.
\end{proof}

\begin{corollary}\label{l:LD} Let $1\leq m\leq n-1$ be integers, let $s\in[m,n]$, and let $\mu$ be a Radon measure on $\RR^n$. If $\mu$ is carried by $(m/s)$-H\"older $m$-cubes, then $$\liminf_{r\downarrow 0} \frac{\mu(B(x,r))}{r^s}>0\quad\text{ for $\mu$-a.e.~$x\in\RR^n$}.$$\end{corollary}

\begin{proof} Let $\mu=\mu_{m\rightarrow s}+\mu^\perp_{m\rightarrow s}$ denote the decomposition of $\mu$ given by Proposition \ref{p:decomp}, where $\mu_{m\rightarrow s}$ is carried by $(m/s)$-H\"older $m$-cubes and $\mu^\perp_{m\rightarrow s}$ is singular to $(m/s)$-H\"older $m$-cubes. Then $\mu$ is carried by $(m/s)$-H\"older $m$-cubes if and only if $\mu^{\perp}_{m\rightarrow s}(\RR^n)=0$. Hence $$\mu\left(\left\{x\in\RR^n:\liminf_{r\downarrow 0}\frac{\mu(B(x,r))}{r^s}=0\right\}\right)=\underline{\mu}^s_{\,0}(\RR^n) \leq \mu^{\perp}_{m\rightarrow s}(\RR^n)=0,$$ where the inequality holds by Lemma \ref{l:pu}. Thus, the lower $s$-density is positive at $\mu$-almost every $x\in\RR^n$.\end{proof}

We now switch focus to the second half of Theorem \ref{thm:A}---points of rapidly infinite density of a Radon measure are carried by H\"older curves. To that end, for every Radon measure $\mu$ on $\RR^n$ and $1\leq s<\infty$, define the quantity $$S^s(\mu,x):=\int_0^1 \frac{r^s}{\mu(B(x,r))}\,\frac{dr}{r}\in[0,\infty]\quad\text{for all }x\in\RR^n.$$ Note that if $S^s(\mu,x)<\infty$, then $\lim_{r\downarrow 0}r^{-s}\mu(B(x,r))=\infty$.

\begin{lemma}\label{l:Ss} Let $\mu$ be a Radon measure on $\RR^n$. Given parameters $1\leq s \leq n$, $0\leq N<\infty$, $1\leq P<\infty$, $\theta>0$, and $x_0\in\RR^n$, consider the sets  $$A:=\left\{x\in B(x_0,1/2): S^s(\mu,x)\leq N \text{ and } \mu(B(x,3r))\leq P\mu(B(x,r))\text{ for all }0<r\leq 1\right\}$$ and $$A':=\left\{x\in A: \mu(A\cap B(x,r)) \geq \theta \mu(B(x,r))\text{ for all }0<r\leq 1\right\}.$$ Then there exists a tree of sets $\mathcal{T}$ whose elements are balls centered in $A'$ such that $$\leaves(\mathcal{T})\supseteq A'$$ and $$\sum_{E\in\mathcal{T}} (\diam E)^s \leq \frac{2^{s+1}sNP\mu(A)}{\theta}<\infty.$$\end{lemma}

\begin{proof} For each $k\geq 0$, let $A'_k$ be a maximal $2^{-k}$ separated subset of $A'$ and define $$\mathcal{T}_k:=\{B(y,2^{-k}):y\in A'_k\}.$$ For each $k\geq 1$ and each $B(y,2^{-k})\in \mathcal{T}_k$, choose $y^\uparrow \in A_{k-1}'$ such that $|y-y^\uparrow|< 2^{-(k-1)}$ and set $B(y,2^{-k})^\uparrow = B(y^\uparrow, 2^{-(k-1)})$. Then $\mathcal{T}=\bigcup_{k=0}^\infty\mathcal{T}_k$ is a tree of sets in the sense of Definition \ref{def:tree} and $\leaves(\mathcal{T})\supseteq A'$.

To estimate the sum of diameters, note that \begin{equation}\begin{split}\label{e:Abound} N\mu(A) \geq \int_A S^s(\mu,x)\,d\mu(x)&=\int_A\int_0^1 \frac{r^s}{\mu(B(x,r))}\,\frac{dr}{r}d\mu(x)\\ &=\sum_{k=0}^\infty \int_{2^{-(k+1)}}^{2^{-k}} r^s\int_A \frac{1}{\mu(B(x,r))}\,d\mu(x)\frac{dr}{r},\end{split}\end{equation} where we used Tonelli's theorem to exchange the order of integration. Our task will be to bound the right hand side of \eqref{e:Abound} from below by a constant times $\sum_{E\in\mathcal{T}}(\diam E)^s$. To that end, fix an integer $k\geq 0$ and $r\in[2^{-(k+1)},2^{-k}]$. Since $A'_k$ is a $2^{-k}$ separated set in $A'$ and $A'\subseteq A$, it follows that $$\int_A \frac{1}{\mu(B(x,r))}d\mu(x) \geq \sum_{y\in A'_k} \int_{A\cap B(y,2^{-(k+1)})} \frac{1}{\mu(B(x,2^{-k}))}\,d\mu(x).$$ By the triangle inequality, $B(x,2^{-k})\subseteq B(y, 3\cdot 2^{-(k+1)})$ whenever $x\in B(y, 2^{-(k+1)})$. Hence  $$\mu(B(x,2^{-k})) \leq \mu(B(y,3\cdot 2^{-(k+1)})) \leq P \mu(B(y,2^{-(k+1)})),$$ where the $P$ is the doubling parameter. Thus,
\begin{equation}\label{e:Abound2}\int_A \frac{1}{\mu(B(x,r))}d\mu(x) \geq \sum_{y\in A'_k} \frac{1}{P}\int_{A\cap B(y,2^{-(k+1)})} \frac{1}{\mu(B(y,2^{-(k+1)}))}\,d\mu(x) \geq \sum_{y\in A_k'} \frac{\theta}{P}.\end{equation} We have shown that \eqref{e:Abound2} holds for all integers $k\geq 0$ and $r\in[2^{-(k+1)},2^{-k}]$. Combining \eqref{e:Abound} and \eqref{e:Abound2}, we obtain $$\frac{PN\mu(A)}{\theta} \geq \sum_{k=0}^\infty \sum_{y\in A'_k} \int_{2^{-(k+1)}}^{2^{-k}} r^s\,\frac{dr}{r}=\frac{1-2^{-s}}{s}\sum_{k=0}^\infty\sum_{y\in A_k'} (2^{-k})^s \geq \frac{1}{2s} \sum_{E\in\mathcal{T}} \left(\frac{\diam E}{2}\right)^s,$$ as desired.\end{proof}

The second half of Theorem \ref{thm:A} is contained in the following theorem.

\begin{theorem} \label{t:Ss} Let $\mu$ be a Radon measure on $\RR^n$ and let $1\leq s\leq n$. Then $$\underline{\mu}^s_{\,\infty}:=\mu\res\left\{x\in\RR^n: S^s(\mu,x)<\infty\text{ and } \limsup_{r\downarrow 0}\frac{\mu(B(x,2r))}{\mu(B(x,r))}<\infty \right\}$$ is carried by $(1/s)$-H\"older curves. Moreover, there exist countably many $(1/s)$-H\"older curves $\Gamma_i\subseteq\RR^n$ and compact sets $K_i\subseteq \Gamma_i$ with $\Haus^s(K_i)=0$ such that $\underline{\mu}^s_{\,\infty}(\RR^n\setminus\bigcup_i K_i)=0.$ \end{theorem}

\begin{proof} By writing the set $$\left\{x\in\RR^n: S^s(\mu,x)<\infty\text{ and }\limsup_{r\downarrow 0}\frac{\mu(B(x,2r))}{\mu(B(x,r))}<\infty\right\}$$ as a countable union of sets of the form $$A:=\left\{x\in B(x_0,1/2): S^s(\mu,x)\leq N \text{ and } \mu(B(x,3r))\leq P\mu(B(x,r))\text{ for all }0<r\leq 1\right\},$$ we see that it suffices to prove $\mu\res A$ is carried by compact $\Haus^s$ null subsets of $(1/s)$-H\"older curves for each choice of parameters $0\leq N<\infty$, $1\leq P<\infty$, and $x_0\in\RR^n$. Fix values for $N$, $P$, and $x_0$, and for all $\theta\in(0,1)$ define $$A'_\theta:=\left\{x\in A: \mu(A\cap B(x,r)) \geq \theta \mu(B(x,r))\text{ for all }0<r\leq 1\right\}.$$ By a standard density theorem for Radon measures (e.g.~ see \cite[Corollary 2.14]{Mattila}), $$\lim_{r\downarrow 0} \frac{\mu(A\cap B(x,r))}{\mu(B(x,r))}=1\quad\text{for $\mu$-a.e. $x\in A$}.$$ Note that $$\left\{x\in A: \lim_{r\downarrow 0} \frac{\mu(A\cap B(x,r))}{\mu(B(x,r))}=1\right\}\subseteq \bigcup_{k=1}^\infty A'_{1/k}.$$ Hence $\mu(A\setminus \bigcup_{k=1}^\infty A'_{1/k})=0$, and so to prove $\mu\res A$ is carried by compact $\Haus^s$ null subsets of $(1/s)$-H\"older curves, it suffices to prove $\mu\res A'_\theta$ has that same property for all $\theta\in(0,1)$. Fix $\theta\in(0,1)$ and apply Lemma \ref{l:Ss} to find a tree of sets $\mathcal{T}$ such that $\leaves(\mathcal{T})\supseteq A'_\theta$ and $$\sum_{E\in\mathcal{T}}(\diam E)^s<\infty.$$ By Theorem \ref{t:diam^s}, $\Haus^s(\leaves(\mathcal{T}))=0$ and there exists a $(1/s)$-H\"older curve $\Gamma$ such that $$\Gamma\supseteq\leaves(\mathcal{T})\supseteq A'_\theta.$$ It follows that $$\mu\res A'_\theta(\RR^n\setminus \Gamma)\leq \mu\res A'_\theta(\RR^n\setminus \leaves(\mathcal{T})) \leq \mu\res A'_\theta(\RR^n\setminus A'_\theta)=\mu(A'_\theta\setminus A'_\theta)=0.$$ Thus, $\mu\res A'_\theta$ is carried by a compact $\Haus^s$ null subset ($\leaves(\mathcal{T})$) of a $(1/s)$-H\"older curve. The theorem follows by taking a suitable choice of countably many parameter values.\end{proof}

We now observe that it is possible to remove the doubling condition from Theorem \ref{t:Ss} by working with dyadic density ratios instead of spherical density ratios. For every Radon measure $\mu$ on $\RR^n$ and $1\leq s<\infty$, define the quantity $$S^s_\Delta(\mu,x):= \sum_{Q\in\Delta} \frac{(\diam Q)^s}{\mu(Q)}\chi_Q(x)\in[0,\infty]\quad\text{for all }x\in\RR^n,$$ where $\Delta$ denotes a system of \emph{half-open} dyadic cubes in $\RR^n$ of side length at most $1$.

The following localization lemma is a particular instance of \cite[Lemma 5.6]{BS3}.

\begin{lemma}\label{l:Qlocal} Let $\mu$ be a Radon measure on $\RR^n$. Given a cube $Q_0\in\Delta$ of side length 1 such that $\eta:=\mu(Q_0)>0$, $N<\infty$, and $0<\varepsilon<1/\eta$, there exists a subtree $\mathcal{G}$ of the tree of dyadic cubes $\{Q\in\Delta:Q\subseteq Q_0\}$ with the following properties. \begin{enumerate}
\item The sets $A:=\{x\in Q_0: S^s_\Delta(\mu,x)<N\}$ and $A':=A \cap \leaves(\mathcal{G})$ have comparable measure: $$\mu(A') \geq (1-\varepsilon\eta)\mu(A).$$
\item The tree $\mathcal{G}$ is $s$-summable: $$\mathcal{S}^s(\mathcal{G})=\sum_{Q\in\mathcal{G}} (\diam Q)^s<\infty.$$
\end{enumerate}\end{lemma}

\begin{proof} Either modify the proof of \cite[Lemma 3.2]{BS2} or apply \cite[Lemma 5.6]{BS3} with $\mathcal{T}:=\{Q\in\Delta:Q\subseteq Q_0\}$, the tree of dyadic cubes contained in $Q_0$, and the function $b(Q):=(\diam Q)^s$ for all $Q\in\mathcal{T}$.\end{proof}

Using Lemma \ref{l:Qlocal} in conjunction with Theorem \ref{t:diam^s}, one can verify the following variant of Theorem \ref{t:Ss}. The case $s=1$ first appeared in \cite[Theorem 3.1]{BS2}.

\begin{theorem}\label{t:fast} Let $\mu$ be a Radon measure on $\RR^n$ and let $1\leq s\leq n$. Then $$\nu:=\mu \res\left\{x\in\RR^n: S^s_\Delta(\mu,x)<\infty\right\}$$ is carried by $(1/s)$-H\"older curves. Moreover, there exist countably many $(1/s)$-H\"older curves $\Gamma_i\subseteq\RR^n$ and compact sets $K_i\subseteq \Gamma_i$ with $\Haus^s(K_i)=0$ such that $\nu(\RR^n\setminus\bigcup_i K_i)=0.$\end{theorem}

\section{Densities and Assouad dimension (Proof of Theorems \ref{thm:B} and \ref{thm:C})}
\label{s:proof-BC}

Theorems \ref{thm:B} and \ref{thm:C} follow from the bi-Lipschitz and H\"older parameterization theorems in \S\ref{s:surfaces} and the following connection between Hausdorff densities and Assouad dimension.

\begin{lemma}\label{lem:1}
Let $\mu$ be a Radon measure on $\RR^n$ and let $t\in[0,n]$. If $E\subseteq\RR^n$ and
\begin{equation}\label{eq:reg-Assouad}
a r^t\leq \mu(B(x,r))\leq b r^t\quad\text{for all $0<r\leq 2\diam E$ and $x\in E$},
\end{equation}
for some constants $0<a\leq b<\infty$, then the Assouad dimension of  $E$ is at most $t$. Additionally, if $\mu(\R^n\setminus E) = 0$, then the Assouad dimension of $E$ is $t$.\end{lemma}

\begin{proof} Let $A\subseteq E$ be bounded and let $\d\in(0,1)$. Consider the cover $\mathcal{B}$ of $A$ by closed balls of diameter $\delta\diam A$ centered in $A$; that is, $$\mathcal{B}=\left\{B\left(x,\tfrac12\delta\diam A\right):x\in A\right\}.$$ By the Besicovitch covering theorem (see e.g.~\cite[Theorem 2.7]{Mattila}), there exist a positive integer $Q=Q(n)$ and disjoint subfamilies $\mathcal{B}_1,\dots,\mathcal{B}_Q$ of $\mathcal{B}$ such that $$\mathcal{B}\subseteq\bigcup_{i=1}^Q\mathcal{B}_i.$$ For each $1\leq i\leq Q$, we have $$\card\mathcal{B}_i\cdot a\left(\tfrac12\delta\diam A\right)^t \leq \sum_{B\in\mathcal{B}_i}\mu(B) = \mu\left(\bigcup\mathcal{B}_i\right)\leq \mu(B(x_i, 2\diam A)) \leq b (2\diam A)^t,$$ where $x_i$ denotes an arbitrarily chosen point in $A\cap \bigcup\mathcal{B}_i$. Hence $\card\mathcal{B}_i \leq \delta^{-t} 2^t(b/a)$ for all $1\leq i\leq Q$. Thus, $\mathcal{B'}=\bigcup_{i=1}^Q\mathcal{B}_i$ is a cover of $A$ by sets of diameter $\delta\diam A$ with $$\card\mathcal{B}' \leq C(n,t,a,b) \delta^{-t}.$$ We have shown the set $E$ is $(C,t)$-homogenous (see Definition \ref{def:assouad}), where $C=C(n,t,a,b)$. Therefore, the Assouad dimension of $E$ is at most $t$.

Suppose in addition to \eqref{eq:reg-Assouad} that $\mu(\R^n\setminus E) = 0$. Consider $A=E\cap B(x,r)$ for some fixed $x\in E$ and $0<r<\diam E$. Fix $\d\in(0,1)$ and let $\{A_1,\dots,A_k\}$ be any cover of $A$ with $A_i\subset E$ and $\diam{A_i} \leq \d\diam{A}$. Let $V$ be a maximal subset of $A$ such that $|v-v'|\geq 2\d\diam{A}$ for all distinct $v,v' \in A$. Cleary, $\card V \leq k$. By maximality of $V$, the collection $\{B(v,4\d r) : v\in V\}$ covers $A$, and thus,
\[ar^t \leq \mu(B(x,r)) = \mu(A)  \leq \sum_{v\in V}\mu(B(v,4\d r)) \leq \card V \cdot b 4^t r^t \d^t,\] where the equality holds since $\mu(\RR^n\setminus E)=0$.
In particular, $k\geq C'(n,t,a,b) \d^{-t}$. Because $\delta\in(0,1)$ was arbitrary, $E$ is not $\beta$-homogeneous for any $\beta <t$. Therefore, the Assouad dimension of $E$ is exactly $t$.
\end{proof}

\begin{corollary}\label{cor:1} Let $\mu$ be a Radon measure on $\RR^n$ and let $t\in[0,n]$. Then $$\mu^t_+:=\mu\res\left\{x\in\RR^n: 0<\liminf_{r\downarrow 0}\frac{\mu(B(x,r))}{r^t}\leq \limsup_{r\downarrow 0}\frac{\mu(B(x,r))}{r^t}<\infty\right\}$$ is carried by sets of Assouad dimension at most $t$.
\end{corollary}

We are ready to prove Theorems \ref{thm:B} and \ref{thm:C}.

\begin{proof}[Proof of Theorem \ref{thm:B}] Let $\mu$ be a Radon measure on $\RR^n$ and let $t\in[0,1)$. By Corollary \ref{cor:1}, we can find countably many sets $E_i\subseteq\RR^n$ with $\dim_A{E_i}\leq t<1$ such that $$\mu^t_+\left(\RR^n\setminus\bigcup_i E_i\right)=0.$$ For each set $E_i$, there exists a bi-Lipschitz embedding $f_i:\RR\rightarrow\RR^n$ such that $E_i \subseteq f_i(\RR)$ by Corollary \ref{cor:bilip}. Hence $$\mu^t_+\left(\RR^n\setminus\bigcup_i\bigcup_{k\in\Z}f_i([k,k+1])\right)=\mu^t_+\left(\RR^n\setminus\bigcup_i f_i(\RR)\right)=0.$$ Therefore, $\mu^t_+$ is carried by bi-Lipschitz curves. \end{proof}

\begin{proof}[Proof of Theorem \ref{thm:C}] Repeat the proof of Theorem \ref{thm:B} \emph{mutatis mutandis}, using Theorem \ref{thm:param} in place of Corollary \ref{cor:bilip}.\end{proof}

\addtocontents{toc}{\protect\vspace{10pt}}

\appendix

\section{Decomposition of $\sigma$-finite measures}
\label{a:A}

The following definition encodes commonly used definitions of countably rectifiable and purely unrectifiable measures, including the variants in Definition \ref{def:s-rect}.

\begin{definition}\label{d:carry} Let $(\mathbb{X},\mathcal{M})$ be a measurable space, let $\mathcal{N}\subseteq\mathcal{M}$ be a nonempty collection of measurable sets, and let $\mu$ be a measure defined on $(\mathbb{X},\mathcal{M})$. We say that $\mu$ is \emph{carried by} $\mathcal{N}$ provided there exists a countable family $\{\Gamma_i:i\geq 1\}\subseteq\mathcal{N}$ of sets with $$\mu\left(\mathbb{X}\setminus\bigcup_{i=1}^\infty\Gamma_i\right)=0.$$ We say that $\mu$ is \emph{singular to} $\mathcal{N}$ if $\mu(\Gamma)=0$ for every $\Gamma\in\mathcal{N}$. \end{definition}

The ``correctness" of Definition \ref{d:carry} is partially justified by the following proposition, which should be considered a standard exercise in measure theory. The proof is a slight variation of \cite[Proposition 1.1]{BS3} (or \cite[Theorem 15.6]{Mattila}), which is specialized to the decomposition of Radon measures (sets) in $\RR^n$ into countably $m$-rectifiable and purely $m$-unrectifiable components. We present details for the convenience of the reader.

\begin{proposition}[Decomposition] \label{p:decomp} Let $(\mathbb{X},\mathcal{M})$ be a measurable space and let $\mathcal{N}\subseteq\mathcal{M}$ be a nonempty collection of sets. If $\mu$ is a $\sigma$-finite measure on $(\XX,\mathcal{M})$, then $\mu$ can be written uniquely as \begin{equation}\label{e:decomp}\mu=\mu_\mathcal{N}+\mu_\mathcal{N}^\perp,\end{equation} where $\mu_\mathcal{N}$ is a measure on $(\mathbb{X},\mathcal{M})$ that is carried by $\mathcal{N}$ and $\mu^\perp_\mathcal{N}$ is a measure on $(\mathbb{X},\mathcal{M})$ that is singular to $\mathcal{N}$.\end{proposition}

\begin{proof}Let $\widetilde{\mathcal{N}}$ denote the collection of finite unions of sets in $\mathcal{N}$.  Given a $\sigma$-finite measure $\mu$ on $(\mathbb{X},\mathcal{M})$, expand $\XX=\bigcup_{j=1}^\infty X_j$, where $$X_1\subseteq X_2\subseteq\cdots$$ is an increasing chain of sets in $\mathcal{M}$ with $\mu(X_j)<\infty$ for all $j\geq 1$. For each $j\geq 1$, define $$M_j := \sup_{N\in\widetilde{\mathcal{N}}} \mu(X_j\cap N)\leq \mu(X_j)<\infty.$$ By the approximation property of the supremum, we may choose a sequence $(N_j)_{j=1}^\infty$ of sets in $\widetilde{\mathcal{N}}$ such that $\mu(X_j\cap N_j)>M_j-1/j$ for all $j\geq 1$. Fix any such $(N_j)_{j=1}^\infty$ and define $$\mu_{\mathcal{N}}:=\mu\res \bigcup_{j=1}^\infty N_j\quad\text{and}\quad\mu^\perp_{\mathcal{N}}:= \mu\res \mathbb{X}\setminus \bigcup_{j=1}^\infty N_j.$$ Then $\mu_\mathcal{N}$ and $\mu^\perp_\mathcal{N}$ are measures on $(\mathbb{X},\mathcal{M})$ with $\mu=\mu_\mathcal{N}+\mu_\mathcal{N}^\perp$ and it is clear that $\mu_{\mathcal{N}}$ is carried by $\mathcal{N}$.

To see that $\mu^\perp_\mathcal{N}$ is singular to $\mathcal{N}$, assume for contradiction that $\mu^\perp_{\mathcal{N}}(S)>0$ for some $S\in\mathcal{N}$. First pick an index $j_0$ such that $\mu(X_{j_0}\cap S)>0$. Next, pick $j\geq j_0$ sufficiently large such that $\mu(X_{j_0}\cap S)>1/j$. Note that $T:=N_j\cup S\in\widetilde{\mathcal{N}}$, since $N_j\in\widetilde{\mathcal{N}}$ and $S\in\mathcal{N}$. It follows that $$M_j\geq \mu(X_j\cap T) \geq \mu_\mathcal{N}(X_j\cap N_j) + \mu^\perp_\mathcal{N}(X_j\cap S)>(M_j-1/j)+1/j=M_j,$$ where in the last inequality we used the fact that $X_{j_0}\subseteq X_j$. We have a reached a contradiction. Therefore, $\mu^\perp_\mathcal{N}$ is singular to $\mathcal{N}$.

Next we want to show that the decomposition of $\mu$ as the sum of a measure that is carried by $\mathcal{N}$ and a measure that is singular to $\mathcal{N}$ is unique.
Suppose that $\mu=\mu_c+\mu_s$, where $\mu_c$ and $\mu_s$ are measures such that $\mu_c$ is carried by $\mathcal{N}$ and $\mu_s$ is singular to $\mathcal{N}$. To show that $\mu_c=\mu_\mathcal{N}$ and $\mu_s=\mu^\perp_\mathcal{N}$, it suffices to prove the former. Suppose for contradiction that $\mu_c(A)<\mu_\mathcal{N}(A)$ for some $A\in\mathcal{M}$. Replacing $A$ with $A\cap X_j$ for $j$ sufficiently large, we may assume without loss of generality that $\mu_{\mathcal{N}}(A)<\infty$. Since $\mu_c$ and $\mu_\mathcal{N}$ are both carried by $\mathcal{N}$, we can find a set $N$, which is a countable union of sets in $\mathcal{N}$ such that $$\mu_c(A\cap N)=\mu_c(A)<\mu_\mathcal{N}(A)=\mu_\mathcal{N}(A\cap N).$$ Then $\mu_s(A\cap N)=\mu(A\cap N)-\mu_c(A\cap N) > \mu(A\cap N)-\mu_\mathcal{N}(A\cap N)=\mu^\perp_\mathcal{N}(A\cap N)=0.$ This contradicts that $\mu_s$ is singular to $\mathcal{N}$. Therefore, $\mu_c=\mu_\mathcal{N}$, and thus, $\mu_s=\mu^\perp_\mathcal{N}$.
\end{proof}

\begin{example} Let $\mu$ and $\nu$ be measures on a measurable space $(\mathbb{X},\mathcal{M})$, and let $$\mathcal{N}:=\{A\in\mathcal{M}:\nu(A)=0\}$$ denote the null sets of $\nu$. If $\mu$ is $\sigma$-finite, then by Proposition \ref{p:decomp}, the measure $\mu$ can be uniquely expanded $\mu=\mu_\mathcal{N}+\mu_{\mathcal{N}}^\perp$, where $\mu_{\mathcal{N}}$ is carried by null sets of $\nu$ and $\mu_{\mathcal{N}}^\perp$ is singular to null sets of $\nu$. Thus, writing $\mu_s:=\mu_{\mathcal{N}}$ and $\mu_{ac}:=\mu_{\mathcal{N}}^\perp$, we can decompose $\mu=\mu_{s}+\mu_{ac}$, where $\mu_{s}\perp \nu$ and $\mu_{ac}\ll\nu$.  \end{example}

\bibliography{mcurves}{}

\providecommand{\bysame}{\leavevmode\hbox to3em{\hrulefill}\thinspace}
\providecommand{\MR}{\relax\ifhmode\unskip\space\fi MR }
\providecommand{\MRhref}[2]{%
  \href{http://www.ams.org/mathscinet-getitem?mr=#1}{#2}
}
\providecommand{\href}[2]{#2}
\begin{thebibliography}{ENV16}

\bibitem[AO17]{AO-curves}
Giovanni Alberti and Martino Ottolini, \emph{On the structure of continua with
  finite length and {G}olab's semicontinuity theorem}, Nonlinear Anal.
  \textbf{153} (2017), 35--55. \MR{3614660}

\bibitem[AS18]{AS-TST}
Jonas Azzam and Raanan Schul, \emph{An analyst's traveling salesman theorem for
  sets of dimension larger than one}, Math. Ann. \textbf{370} (2018), no.~3-4,
  1389--1476. \MR{3770170}

\bibitem[Ass83]{Assouad2}
Patrice Assouad, \emph{Plongements lipschitziens dans {${\bf R}^{n}$}}, Bull.
  Soc. Math. France \textbf{111} (1983), no.~4, 429--448. \MR{763553}

\bibitem[AT15]{AT15}
Jonas Azzam and Xavier Tolsa, \emph{Characterization of {$n$}-rectifiability in
  terms of {J}ones' square function: {P}art {II}}, Geom. Funct. Anal.
  \textbf{25} (2015), no.~5, 1371--1412. \MR{3426057}

\bibitem[Bes28]{Bes28}
A.~S. Besicovitch, \emph{On the fundamental geometrical properties of linearly
  measurable plane sets of points}, Math. Ann. \textbf{98} (1928), no.~1,
  422--464. \MR{1512414}

\bibitem[Bes38]{Bes38}
A.~S. Besicovitch, \emph{On the fundamental geometrical properties of linearly
  measurable plane sets of points ({II})}, Math. Ann. \textbf{115} (1938),
  no.~1, 296--329. \MR{1513189}

\bibitem[BH04]{BoHei}
Mario Bonk and Juha Heinonen, \emph{Smooth quasiregular mappings with
  branching}, Publ. Math. Inst. Hautes \'Etudes Sci. (2004), no.~100, 153--170.
  \MR{2102699}

\bibitem[BS15]{BS1}
Matthew Badger and Raanan Schul, \emph{Multiscale analysis of 1-rectifiable
  measures: necessary conditions}, Math. Ann. \textbf{361} (2015), no.~3-4,
  1055--1072. \MR{3319560}

\bibitem[BS16]{BS2}
Matthew Badger and Raanan Schul, \emph{Two sufficient conditions for
  rectifiable measures}, Proc. Amer. Math. Soc. \textbf{144} (2016), no.~6,
  2445--2454. \MR{3477060}

\bibitem[BS17]{BS3}
Matthew Badger and Raanan Schul, \emph{{M}ultiscale {A}nalysis of 1-rectifiable
  {M}easures {II}: {C}haracterizations}, Anal. Geom. Metr. Spaces \textbf{5}
  (2017), 1--39. \MR{3627148}

\bibitem[DS91]{DS91}
G.~David and S.~Semmes, \emph{Singular integrals and rectifiable sets in {${\bf
  R}\sp n$}: {B}eyond {L}ipschitz graphs}, Ast\'erisque (1991), no.~193, 152.
  \MR{1113517 (92j:42016)}

\bibitem[DS93]{DS93}
Guy David and Stephen Semmes, \emph{Analysis of and on uniformly rectifiable
  sets}, Mathematical Surveys and Monographs, vol.~38, American Mathematical
  Society, Providence, RI, 1993. \MR{1251061 (94i:28003)}

\bibitem[DS97]{DS}
Guy David and Stephen Semmes, \emph{Fractured fractals and broken dreams},
  Oxford Lecture Series in Mathematics and its Applications, vol.~7, The
  Clarendon Press, Oxford University Press, New York, 1997, Self-similar
  geometry through metric and measure. \MR{1616732}

\bibitem[DT99]{DT99}
G.~David and T.~Toro, \emph{Reifenberg flat metric spaces, snowballs, and
  embeddings}, Math. Ann. \textbf{315} (1999), no.~4, 641--710. \MR{1731465}

\bibitem[DT12]{DT}
Guy David and Tatiana Toro, \emph{Reifenberg parameterizations for sets with
  holes}, Mem. Amer. Math. Soc. \textbf{215} (2012), no.~1012, vi+102.
  \MR{2907827}

\bibitem[ENV16]{ENV}
Nick Edelen, Aaron Naber, and Daniele Valtorta, \emph{Quantitative {R}eifenberg
  theorem for measures}, preprint, \textsf{arXiv:1612.08052}, 2016.

\bibitem[Fal86]{Falconer}
K.~J. Falconer, \emph{The geometry of fractal sets}, Cambridge Tracts in
  Mathematics, vol.~85, Cambridge University Press, Cambridge, 1986. \MR{867284
  (88d:28001)}

\bibitem[Fed47]{Fed47}
Herbert Federer, \emph{The {$(\varphi,k)$} rectifiable subsets of {$n$}-space},
  Trans. Amer. Soc. \textbf{62} (1947), 114--192. \MR{0022594}

\bibitem[Ghi17]{ghinassi}
Silvia Ghinassi, \emph{A sufficient condition for {$C^{1,\alpha}$}
  parametrization}, preprint, \textsf{arXiv:1709.06015}, 2017.

\bibitem[GKS10]{GKS}
John Garnett, Rowan Killip, and Raanan Schul, \emph{A doubling measure on
  {$\mathbb{R} \sp d$} can charge a rectifiable curve}, Proc. Amer. Math. Soc.
  \textbf{138} (2010), no.~5, 1673--1679. \MR{2587452 (2011a:28018)}

\bibitem[Hei01]{Heinonen}
Juha Heinonen, \emph{Lectures on analysis on metric spaces}, Universitext,
  Springer-Verlag, New York, 2001. \MR{1800917}

\bibitem[Hut81]{hutchinson}
John~E. Hutchinson, \emph{Fractals and self-similarity}, Indiana Univ. Math. J.
  \textbf{30} (1981), no.~5, 713--747. \MR{625600}

\bibitem[Jon90]{Jones-TST}
Peter~W. Jones, \emph{Rectifiable sets and the traveling salesman problem},
  Invent. Math. \textbf{102} (1990), no.~1, 1--15. \MR{1069238 (91i:26016)}

\bibitem[Kec95]{Kechris}
Alexander~S. Kechris, \emph{Classical descriptive set theory}, Graduate Texts
  in Mathematics, vol. 156, Springer-Verlag, New York, 1995. \MR{1321597}

\bibitem[Luu98]{Luukk}
Jouni Luukkainen, \emph{Assouad dimension: antifractal metrization, porous
  sets, and homogeneous measures}, J. Korean Math. Soc. \textbf{35} (1998),
  no.~1, 23--76. \MR{1608518}

\bibitem[Mac99]{MM}
Paul MacManus, \emph{Catching sets with quasicircles}, Rev. Mat. Iberoamericana
  \textbf{15} (1999), no.~2, 267--277. \MR{1715408}

\bibitem[Mar54]{Marstrand54}
J.~M. Marstrand, \emph{Some fundamental geometrical properties of plane sets of
  fractional dimensions}, Proc. London Math. Soc. (3) \textbf{4} (1954),
  257--302. \MR{0063439}

\bibitem[Mar61]{Marstrand61}
J.~M. Marstrand, \emph{Hausdorff two-dimensional measure in {$3$}-space}, Proc.
  London Math. Soc. (3) \textbf{11} (1961), 91--108. \MR{0123670}

\bibitem[Mat75]{Mattila75}
Pertti Mattila, \emph{Hausdorff {$m$} regular and rectifiable sets in
  {$n$}-space}, Trans. Amer. Math. Soc. \textbf{205} (1975), 263--274.
  \MR{0357741 (50 \#10209)}

\bibitem[Mat95]{Mattila}
Pertti Mattila, \emph{Geometry of sets and measures in {E}uclidean spaces},
  Cambridge Studies in Advanced Mathematics, vol.~44, Cambridge University
  Press, Cambridge, 1995, Fractals and rectifiability. \MR{1333890 (96h:28006)}

\bibitem[MM88]{MM1988}
Miguel~\'Angel Mart{\'\i}n and Pertti Mattila, \emph{{$k$}-dimensional
  regularity classifications for {$s$}-fractals}, Trans. Amer. Math. Soc.
  \textbf{305} (1988), no.~1, 293--315. \MR{920160}

\bibitem[MM93]{MM1993}
Miguel~\'Angel Mart{\'\i}n and Pertti Mattila, \emph{Hausdorff measures,
  {H}\"older continuous maps and self-similar fractals}, Math. Proc. Cambridge
  Philos. Soc. \textbf{114} (1993), no.~1, 37--42. \MR{1219912}

\bibitem[MM00]{MM2000}
Miguel~Angel Mart{\'\i}n and Pertti Mattila, \emph{On the parametrization of
  self-similar and other fractal sets}, Proc. Amer. Math. Soc. \textbf{128}
  (2000), no.~9, 2641--2648. \MR{1664402}

\bibitem[MR44]{MR44}
A.~P. Morse and John~F. Randolph, \emph{The {$\phi$} rectifiable subsets of the
  plane}, Trans. Amer. Math. Soc. \textbf{55} (1944), 236--305. \MR{0009975}

\bibitem[MT10]{McTy}
John~M. Mackay and Jeremy~T. Tyson, \emph{Conformal dimension}, University
  Lecture Series, vol.~54, American Mathematical Society, Providence, RI, 2010,
  Theory and application. \MR{2662522}

\bibitem[Oki92]{OK-TST}
Kate Okikiolu, \emph{Characterization of subsets of rectifiable curves in
  {${\bf R}\sp n$}}, J. London Math. Soc. (2) \textbf{46} (1992), no.~2,
  336--348. \MR{1182488 (93m:28008)}

\bibitem[Pre87]{Preiss}
David Preiss, \emph{Geometry of measures in {${\bf R}\sp n$}: distribution,
  rectifiability, and densities}, Ann. of Math. (2) \textbf{125} (1987), no.~3,
  537--643. \MR{890162 (88d:28008)}

\bibitem[Rus73]{Rushing}
T.~Benny Rushing, \emph{Topological embeddings}, Academic Press, New
  York-London, 1973, Pure and Applied Mathematics, Vol. 52. \MR{0348752}

\bibitem[RV17]{RV}
Matthew Romney and Vyron Vellis, \emph{Bi-{L}ipschitz embedding of the
  generalized {G}rushin plane into {E}uclidean spaces}, Math. Res. Lett.
  (2017).

\bibitem[Sem03]{semmes-buffalo}
Stephen Semmes, \emph{Where the buffalo roam: infinite processes and infinite
  complexity}, preprint, \textsf{arXiv:0302308}, 2003.

\bibitem[SS05]{ss-reals}
Elias~M. Stein and Rami Shakarchi, \emph{Real analysis}, Princeton Lectures in
  Analysis, vol.~3, Princeton University Press, Princeton, NJ, 2005, Measure
  theory, integration, and Hilbert spaces. \MR{2129625}

\bibitem[Ste70]{Stein}
Elias~M. Stein, \emph{Singular integrals and differentiability properties of
  functions}, Princeton Mathematical Series, No. 30, Princeton University
  Press, Princeton, N.J., 1970. \MR{0290095 (44 \#7280)}

\bibitem[Sto98]{stong}
R.~Stong, \emph{Mapping {$\mathbf{Z}^r$} into {$\mathbf{Z}^s$} with maximal
  contraction}, Discrete Comput. Geom. \textbf{20} (1998), no.~1, 131--138.
  \MR{1626695}

\bibitem[TT15]{TT-rect}
Xavier Tolsa and Tatiana Toro, \emph{Rectifiability via a square function and
  {P}reiss' theorem}, Int. Math. Res. Not. IMRN (2015), no.~13, 4638--4662.
  \MR{3439088}

\bibitem[V\"81]{Vais1}
Jussi V\"ais\"al\"a, \emph{Quasisymmetric embeddings in {E}uclidean spaces},
  Trans. Amer. Math. Soc. \textbf{264} (1981), no.~1, 191--204. \MR{597876}

\bibitem[Vel16]{Vext}
Vyron Vellis, \emph{Extension properties of planar uniform domains}, preprint,
  \textsf{arXiv:1609.08763}, 2016.

\end{thebibliography}
\bibliographystyle{amsbeta}
\end{document}